\documentclass[12pt]{amsart}

\usepackage{amsmath,amsthm,verbatim,amssymb,amsfonts,amscd, graphicx,cite,bbm,url,enumerate,hyperref,comment,xspace,mathtools}
\usepackage{graphics}
\usepackage{todonotes}
\usepackage{pgfplots}
\hypersetup{
	colorlinks=true,
	linkcolor=blue,
	citecolor=black
}
\theoremstyle{plain}
\newtheorem{theorem}{Theorem}[section]

\newtheorem{lemma}[theorem]{Lemma}
\newtheorem{proposition}[theorem]{Proposition}

\newtheorem{conjecture}[theorem]{Conjecture} 
 
\theoremstyle{definition}

\theoremstyle{definition}

\newtheorem{remark}[theorem]{Remark}

\allowdisplaybreaks[2]
\numberwithin{equation}{section}

\pgfplotsset{compat=1.18}

\title[The third moment of the logarithm of zeta]{The third moment of the logarithm of zeta and a twisted pair correlation conjecture}
\date{\today}
\author{Alessandro Fazzari}
\address{D\'epartement de math\'ematiques et de statistique, Universit\'e de Montr\'eal. CP 6128, succ. Centre-ville. Montreal, QC H3C 3J7, Canada}
\email{alessandro.fazzari@umontreal.ca}
\author{Maxim Gerspach}
\address{Department of Mathematics, KTH Royal Institute of Technology, Lindstedtsvägen 25, Stockholm, Sweden}
\email{maxim.gerspach@gmail.com}

\begin{document}

\begin{abstract}
    We prove precise conditional estimates for the third moment of the logarithm of the Riemann zeta function, refining what is implied by the Selberg central limit theorem, both for the real and imaginary parts. These estimates match predictions made in work of Keating and Snaith. We require the Riemann Hypothesis, a conjecture for the triple correlation of Riemann zeros and another ``twisted'' pair correlation conjecture which explains the interaction of a prime power with Montgomery's pair correlation function. We believe this to be of independent interest, and devote substantial effort to its justification. Namely, we prove this conjecture on a certain range unconditionally, and on a larger range under the assumption of a variant of the Hardy-Littlewood conjecture with good uniformity.
\end{abstract}

\maketitle

\section{Introduction}

In the seminal work of Keating and Snaith \cite{KS1}, the authors provide a precise conjecture for the moments of the Riemann zeta function. More specifically, they predict that 
\begin{equation}\label{KSConj}
    \frac{1}{T} \int_T^{2 T} \left|\zeta \left(\tfrac{1}{2} + i t \right) \right|^{2s} \, dt \sim g(s) a(s) (\log T)^{s^2}
    \qquad (\Re s>-\tfrac{1}{2}).
\end{equation}
Here,
\[ a(s) = \prod_p \left( (1-1/p)^{s^2} \sum_{m \ge 0} \left( \frac{\Gamma(s+m)}{m! \Gamma(s)} \right)^2 \frac{1}{p^m} \right)  \]
is an absolutely convergent Euler product and
\[ g(s) = \frac{G^2(s+1)}{G(2s+1)} \]
is a ratio of Barnes $G$-functions.
While the Euler product part was fairly well-understood before, the random matrix part $g(s)$ was previously of a much more mysterious nature. 

Roughly speaking, the values of $g(s)$ were derived as follows.
Denote by $U(N)$ the set of unitary $N \times N$ matrices and by $\mathbb{E}^{U(N)}$ the expected value with respect to the probability Haar measure on $U(N)$. For a unitary matrix $U$, write
\[ Z(U,\theta) := \det \left( 1 - U e^{-i \theta} \right) \]
for the characteristic polynomial of $U$ evaluated at $e^{-i \theta}$. Then they prove that we have
\begin{equation}
    M_N(s) := \mathbb{E}^{U(N)}[ |Z(U,\theta)|^{2s} ] = \prod_{j=1}^N \frac{\Gamma(j) \Gamma(j+2s)}{\Gamma(j+s)^2}
    \qquad (\Re s>-\tfrac{1}{2}).
\end{equation}
In the large $N$ limit, one can then show that
\[ \lim_{N \to \infty} N^{-s^2} M_N(s) = g(s). \]
In other words, $M_N(s)$ is well-approximated by $g(s) N^{s^2}$, and with the translation $N = \log \frac{T}{2 \pi}$ this yields the random matrix part of \eqref{KSConj}.

In the context of the logarithm of the Riemann zeta function on the critical line, a celebrated result of Selberg \cite{SelbergCLT,SelbergOandN} tells us the following distributional result. Let $T$ be a (large) real parameter, and $\tau$ be a uniformly distributed random variable on $[T,2T]$. Then we have convergence in law
\[ \frac{\log \zeta \left( \frac{1}{2} + i \tau \right)}{\sqrt{\log \log T}} \overset{d}{\rightarrow} \mathcal{N}^{\mathbb{C}}(0,1), \]
a standard complex Gaussian. Selberg (see also Tsang \cite{Tsang1}) proved more precisely that, assuming the Riemann Hypothesis (RH), for any positive integer $k$ we have
\[ \mathbb{E} \left[ \Re \log \zeta \Big( \frac{1}{2} + i \tau \Big)^{2 k} \right] = \mu_k \left( \frac{1}{2} \log \log T \right)^{ k} \left( 1 + O_k \left( \frac{1}{\log \log T} \right) \right) \]
where $\mu_k = \frac{(2k)!}{k! 2^k}$ denotes the $2k$-th Gaussian moment. Moreover, he proved a similar estimate for the imaginary part, and a slightly weaker estimate unconditionally.

Based on this central limit theorem, one might be led to believe that the real and imaginary parts of $\log \zeta$ follow essentially the same distribution. While this is true at a leading order, Odlyzko observed numerically that this does not seem to be the case on a finer level; see e.g. \cite[p. 51]{Odlyzko}. A perhaps slightly lesser-known conjecture made by Keating and Snaith makes this observation more precise. 
%
Setting $N = \log \frac{T}{2 \pi}$, they conjecture \cite[(97)]{KS1} that
\[ \mathbb{E} \left[ \left( \Re \log \zeta \Big( \frac{1}{2} + i \tau \Big) \right)^k \right] \sim \frac{d^k}{ds^k} [ M_N(s/2) a(s/2)] \Big|_{s = 0}, \]
and it seems conceivable that this should hold to a higher level of precision.
They provide a similar conjecture for the imaginary part (see \cite[(98)]{KS1}), which exhibits a different behaviour. In particular, it implies that all the odd moments of $\Im\log\zeta(\frac{1}{2}+i\tau)$ are of size $o(1)$.
%

Specialising to the second moment, this would suggest that 
\begin{equation}\label{GoldstonSM} \mathbb{E} \left[ \left( \Re \log \zeta \Big( \frac{1}{2} + i \tau \Big) \right)^2 \right] = \frac{1}{2} \log \log T + \frac{\gamma + 1}{2} + \frac{1}{2}\sum_{p \text{ prime}} \sum_{m \ge 2} \frac{1-m}{m^2} \frac{1}{p^m} + o(1) \end{equation}
and the same estimate for the imaginary part. This was proved by Goldston \cite{Goldston1} for the imaginary part assuming RH and the pair correlation conjecture (compare Conjecture \ref{PCC}). For the real part, the same estimate was obtained much more recently by Lugar, Milinovich and Quesada-Herrera \cite{LMQ} under the same assumptions.

The main objective of this work is to prove an analogous conditional estimate for the third moment of the real and imaginary parts of $\log \zeta$.

For brevity, let us denote 
$c_Z := - \frac{\pi^2}{4}$ 
and
\[ c_P := \frac{3}{4} \sum_{p, m \ge 2} \frac{1}{m p^m} \sum_{k + \ell = m} \frac{1}{k \ell}. \]
Note that $c_Z = \frac{1}{8}M_N'''(0)$ and $c_P = \frac{1}{8}a'''(0)$.
We have the following

\begin{theorem}\label{MainTheorem}
    Assume RH, the pair and triple correlation Conjectures \ref{PCC} and \ref{TCC2}, and the ``twisted pair correlation'' Conjecture \ref{TPCC}. Then we have
    \begin{align}
    \label{TheoremRP}
        M_3^{\Re}(T) &:= \frac{1}{T} \int_T^{2 T} \left( \Re \log \zeta(1/2+it) \right)^3 \, dt = c_P + c_Z + O \left( \frac{1}{\log T} \right) \\
    \intertext{and}
    \label{TheoremIP}
        M_3^{\Im}(T) &:= \frac{1}{T} \int_T^{2 T} \left( \Im \log \zeta(1/2+it) \right)^3 \, dt = O \left( \frac{1}{\log T} \right).
    \end{align}
\end{theorem}

Before we delve deeper into the assumptions of this theorem, let us first discuss the requirements for the second moment made in \cite{Goldston1,LMQ}. There, the authors require RH and a conjecture concerning the pair correlation of zeros of the Riemann zeta function. 

Recall that Montgomery \cite{Montgomery1} studied the pair correlation through the function
\begin{equation}\label{FDefin}
    F(\alpha) := \left( \frac{T \log T}{2 \pi} \right)^{-1} \sum_{T < \gamma, \gamma' \le 2 T} T^{i \alpha (\gamma - \gamma')} \omega(\gamma-\gamma'),
\end{equation}
where $\alpha\in\mathbb R$ and
\[ \omega(x) = \frac{4}{4+x^2}. \]
The function $F$ is essentially the Fourier transform of the distribution function of $\gamma-\gamma'$. Namely, for a wide class of functions $r$, the definition \eqref{FDefin} immediately implies that
\begin{equation}
    \sum_{T < \gamma, \gamma' \le 2 T} r \left( (\gamma-\gamma') \frac{\log T}{2 \pi} \right) \omega(\gamma-\gamma') = \frac{T \log T}{2 \pi} \int_{\mathbb{R}} F(\alpha)  \hat r(\alpha) \, d\alpha,
\end{equation}
where we use the convention
\[ \hat r(\alpha) = \int_{\mathbb{R}} r(t) e^{-2 \pi i \alpha t} \, dt. \]
Montgomery proved that
\[ F(\alpha) = (1+o(1)) T^{-2 |\alpha|} \log T + |\alpha| + o(1) \]
for $|\alpha| < 1-\varepsilon$. He moreover conjectured that one should more generally have
\[ F(\alpha) = (1+o(1)) T^{-2 |\alpha|} \log T + \min \{ |\alpha|, 1\} + o(1) \]
uniformly for $\alpha$ in bounded intervals. Through Fourier transform, this suggests, writing $\tilde\gamma=\frac{\log T}{2\pi}\gamma$, that one should have

\begin{conjecture}[Pair Correlation Conjecture, Montgomery]\label{PCC}
For any continuous, integrable function $r$ such that $\hat r$ is Lipschitz continuous and integrable, we have
    \begin{equation}\begin{split}\notag
        \bigg(\frac{T\log T}{2\pi}\bigg)^{-1} \sum_{T\leq \gamma,\gamma'\leq 2T}  r(\tilde\gamma-\tilde\gamma')
        =\int_{\mathbb R} \hat r(a) \bigg( \delta(a) + \min\{|a|,1\} \bigg)da  + O\bigg(\frac{1}{\log T}\bigg).
    \end{split}\end{equation}
\end{conjecture}

We note that the size $\frac{1}{\log T}$ of the error term makes crucial use of the Lipschitz continuity of $\hat r$, as discussed e.g. in the introduction of Hejhal \cite{Hejhal1}. Assuming this conjecture, the error term in the estimate \eqref{GoldstonSM} of Goldston \cite{Goldston1}, and the analogous one of \cite{LMQ} for the real part, could be improved to $\frac{1}{\log T}$.

Vast generalisations of this conjecture have since been derived for the $n$-point correlations of the zeros, coming from the $n$-point correlations of eigenvalues of the characteristic polynomial of a random unitary matrix. We refer the reader to the seminal work of Rudnick and Sarnak \cite{RS1} for a detailed discussion of this. In there, the authors prove the $n$-point correlations under certain support assumptions on the Fourier transform of the function one integrates against. This bears similarity to the fact that Montgomery managed to prove the estimate for $F(\alpha)$ in the range $|\alpha| < 1 - \varepsilon$. In order to do so, they need in particular to compute the Fourier transform of the $n$-dimensional sine kernel determinant, compare \cite[Theorem 4.1]{RS1}. For $n = 3$, this was previously worked out explicitly by Hejhal \cite[(11)]{Hejhal1}. His work suggests that one should have

\begin{conjecture}[Triple Correlation Conjecture, Hejhal]\label{TCC2}
    For any continuous, integrable function $r$ such that $\hat r$ is Lipschitz continuous and integrable, we have
    \begin{equation}\begin{split}\notag
        \bigg(\frac{T\log T}{2\pi}\bigg)^{-1} \sum_{T\leq\gamma,\gamma',\gamma''\leq 2T} r(\tilde\gamma-\tilde\gamma',\tilde\gamma-\tilde\gamma'')
        = \int_{\mathbb R} \int_{\mathbb R} \hat r(a,b)H(a,b) da\;db + O\bigg(\frac{1}{\log T}\bigg),
    \end{split}\end{equation}
    where
    $$H(a,b) = H_{\delta}(a,b) + H_*(a,b),$$
    with
    \begin{equation}\begin{split}\label{6may.20} 
        H_{\delta}(a,b) &:= \delta(a)\delta(b) + \delta(a)\min\{|b|,1\} 
        + \delta(b)\min\{|a|,1\} + \delta(a+b)\min\{|a|,1\}\\
        H_*(a,b) &:= 2G(a,b) + \min\{|a|,1\} + \min\{|b|,1\} + \min\{|a+b|,1\} -2 ,
    \end{split}\end{equation}
    and 
    \[G(a,b) = \max \left\{ \frac{1}{2}(2-|a|-|b|-|a+b|),0 \right\}.\] 
    \end{conjecture}

For our purposes, it appears to be insufficient to require only this conjecture (and RH) in order to deduce Theorem \ref{MainTheorem}. The reason for this is the fact that we need to understand the following modification to Montgomery's pair correlation, which we refer to as the ``twisted'' pair correlation function.

Let $n$ be a prime power. For $\alpha\in\mathbb R$, we define  
\begin{equation}\label{DefinitionOfFn}
    F_n(\alpha) := -\bigg(\frac{T}{2\pi}\frac{\Lambda(n)}{\sqrt{n}} \bigg)^{-1}
    \sum_{T\leq\gamma,\gamma'\leq 2T} n^{i\gamma} T^{i\alpha(\gamma-\gamma')}\omega(\gamma-\gamma').
\end{equation}
Note that $F_n(\alpha) = F_n(-\alpha-\frac{\log n}{\log T}).$ A priori, it may not be obvious that the normalisation we have chosen turns out to be reasonable. One of the main efforts of the present work is to reason that this is indeed the case, and attempt to evaluate the functions $F_n$.
As another reference point for understanding this function, we would like to point to Landau's formula in the version of Gonek \cite{Gonek1} for sums of $n^{i \gamma}$, see Lemma \ref{GoneksFormula}.

In order to state more precisely what we believe to be the correct estimate for $F_n(\alpha)$, let us first introduce the following notation:
\begin{equation}\begin{split}\notag
        r_1(\alpha,n) &:=  
        \frac{1}{\Lambda(n)} \sum_{m} \frac{\Lambda(mn) \Lambda(m)}{m}\min \bigg \{\frac{m}{T^\alpha},\frac{T^\alpha}{m}\bigg\}^2 \\
        r_2(\alpha,n) &:= 
        \frac{1}{\Lambda(n)}\sum_m \Lambda(m) \Lambda(n/m) \min \bigg\{ \frac{nT^\alpha}{m}, \frac{m}{nT^\alpha} \bigg\}^2.
\end{split}\end{equation}

\begin{conjecture}[Strong Twisted Pair Correlation Conjecture]\label{STPCC}
Fix $\varepsilon>0$. Uniformly for all prime powers $n\leq T^{1-\varepsilon}$, we have
\begin{equation}\notag
        F_n(\alpha) =
        \begin{cases}
            T^{2 \alpha} (\log T+O(1))  + \frac{\log T+O(1)}{(n T^\alpha)^2} 
            - r_2(\alpha,n) + O (\frac{1}{\log T})  & \alpha\in [-\frac{\log n}{\log T},0]\\
            T^{-2\alpha}(\log T+\frac{\log T}{n^2} +O(1)) -r_1(\alpha,n) + O(\frac{1}{\log T})  & \alpha\in (0,1-\frac{\log n}{\log T}) \\
            \min\{1,\frac{\log T}{\Lambda(n)}(\alpha-1+\frac{\log n}{\log T})\} + O(\frac{1}{\log T}) & \alpha\in [1-\frac{\log n}{\log T},\infty)
        \end{cases}
    \end{equation}
\end{conjecture}

At first glance, this may seem overwhelming. Firstly, we would like to point out that, similar to the pair correlation setting, the summand $T^{-2 |\alpha|} \log T$ only matters near $\alpha = 0$ and essentially acts like a $\delta$-distribution when integrating $F_n$ against another function. Similarly the second term, which can be written as $T^{-2 |\alpha + \frac{\log n}{\log T}|} \log T$, acts as a $\delta$-distribution at $- \frac{\log n}{\log T}$.

Next, the summands $r_1$ and $r_2$ are small at most points, and the reader may choose to ignore them at first reading. They only produce spikes near integer multiples of $\frac{\Lambda(n)}{\log T}$, whose sizes are negligible as soon as $\Lambda(n)$ is moderately large. Moreover, these spikes are small enough that they do not contribute when integrating $F_n$ against a sufficiently nice function.

To elaborate on that, for $n=q^a$ prime power, let us denote
$$\mathcal E_n = \mathcal E_n(T) = (a-1)\frac{ \log q}{\log T} +\frac{1}{\log T}.$$
The above will appear as an error term in the following. Note that the first term in the definition of $\mathcal E_n$ vanishes when $n$ is a prime. 

Additionally, let us denote by $\mathfrak{C}(\mathbb{R})$ the set of functions $r \in L^1(\mathbb{R})$ that are continuous, and such that its Fourier transform $\hat r$ is integrable, Lipschitz continuous, and $\hat r' (a) \ll |a|^{-3}$. 

Then we believe that the following should hold.

\begin{conjecture}[Twisted Pair Correlation Conjecture]\label{TPCC}
    Fix $\varepsilon>0$ and let $r \in \mathfrak{C}(\mathbb{R})$.
    Uniformly for all prime powers $n=q^a \le T^{1-\varepsilon}$, we have
    \begin{equation}\begin{split}\notag
        -\bigg(\frac{T}{2\pi} \frac{\Lambda(n)}{\sqrt{n}}\bigg)^{-1} \sum_{T\leq \gamma,\gamma',2T} n^{i\gamma} r(\tilde\gamma-\tilde\gamma')
        =  \int_{\mathbb R} \frac{\hat r(\alpha)+\hat r(-\alpha-\frac{\log n}{\log T})}{2} \Big(\delta(\alpha)+m_n(\alpha)\Big)  d\alpha 
        +O(\mathcal E_n),
    \end{split}\end{equation} 
    where
    $$m_n(\alpha) = 
    \begin{cases}
    1, & \text{if } \alpha<-1-\frac{\Lambda(n)}{\log T}, \\
    \frac{\log T}{\Lambda(n)}(-\alpha-1), & \text{if } -1-\frac{\Lambda(n)}{\log T}\leq \alpha<-1,\\
    0, & \text{if } -1\leq \alpha<1-\frac{\log n}{\log T},\\
    \frac{\log T}{\Lambda(n)}(\alpha-1+\frac{\log n}{\log T}), & \text{if } 1-\frac{\log n}{\log T}\leq \alpha<1-\frac{\log n -\Lambda(n)}{\log T},\\
    1, & \text{if } \alpha\geq 1-\frac{\log n-\Lambda(n)}{\log T}.
    \end{cases}$$
\end{conjecture}

\begin{figure}[b]
\begin{tikzpicture}
\begin{axis}
    [ytick distance=1,
    xtick={-1.3,-1,0.4,0.7},
    xticklabels={\hspace{-1cm}$-1-\frac{\Lambda(n)}{\log T}$,-1,\hspace{-1cm}$1-\frac{\log n}{\log T}$,\hspace{1.3cm}$1-\frac{\log n - \Lambda(n)}{\log T}$},
    height=5cm,
    width =17cm,
    xmin=-1.9,xmax=1.3]
\addplot[line width = 0.3mm , blue,smooth,samples=200,domain=-1.8:-1.3]{1};
\addplot[line width = 0.3mm , blue,smooth,samples=200,domain=-1.3:-1]{-10/3*x-10/3)};
\addplot[line width = 0.3mm , blue,smooth,samples=200,domain=-1:0.4]{0};
\addplot[line width = 0.3mm , blue,smooth,samples=200,domain=0.4:0.7]{10/3*x-4/3};
\addplot[line width = 0.3mm , blue,smooth,samples=200,domain=0.7:1.2]{1};
\end{axis}
\end{tikzpicture}
\caption{The function $m_n(\alpha)$}
\end{figure}
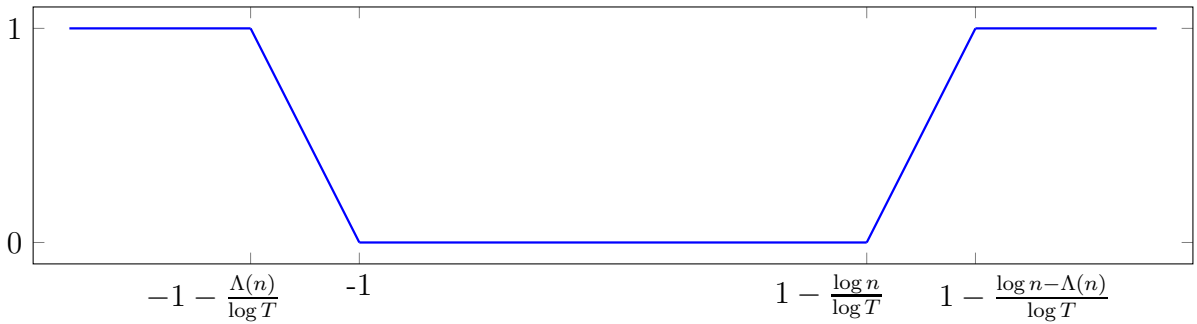

Note that one has $m_n(\alpha) = m_n (-\alpha - \frac{\log n}{\log T})$. Moreover, we would like to point out that if $n$ is a prime then we have \[\frac{\Lambda(n)}{\log T} m_n(\alpha) = H_* \left( \alpha, \frac{\Lambda(n)}{\log T} \right).\] It would be very interesting to understand this identity beyond the fact that it is what is needed in order for the contribution related to the twisted pair correlation to cancel with the triple correlation contribution in the expected way.

Unsurprisingly, Conjecture \ref{TPCC} follows from Conjecture \ref{STPCC} by convolving with a suitable kernel, as in the classical case of Montgomery's pair correlation. In fact, one can show that

\begin{proposition}\label{STPCCimpliesTPCC}
Conjecture \ref{STPCC} implies Conjecture \ref{TPCC}.
\end{proposition}

In support of Conjecture \ref{STPCC}, we prove it in restricted ranges for $\alpha$. Defining $\delta_T := \frac{10 \log \log T}{\log T}$, we first obtain an analogue of Montgomery's formula for $F(\alpha)$ corresponding to $|\alpha|<1-\delta_T$.
Namely, for small values of $\alpha$, we prove the asymptotic formulae of Conjecture \ref{STPCC} assuming RH only.

\begin{proposition}\label{TPC_SmallAlpha}
    Assume RH. 
    Fix $\varepsilon>0$ and let $n=q^a\leq T^{1-\varepsilon}$ be a prime power.
    For $0<\alpha<1-\frac{\log n}{\log T}-\delta_T$, we have
    \begin{equation}\notag
        F_n(\alpha) =
        T^{-2\alpha}\bigg(\log T+\frac{\log T}{n^2} +O(1)\bigg) -r_1(\alpha,n)
        + O\bigg(\frac{1}{\log T}\bigg).
    \end{equation}
    Moreover, for $-\frac{\log n}{\log T}\leq \alpha\leq 0$, we have 
    \begin{equation}\notag
        F_n(\alpha) =
        T^{2 \alpha} (\log T+O(1))  + \frac{\log T+O(1)}{(n T^\alpha)^2} 
        - r_2(\alpha,n) 
        + O\bigg(\frac{1}{\log T}\bigg)  .
    \end{equation}
\end{proposition}
Due to the symmetry of $F_n(\alpha)$ discussed above, Theorem \ref{TPC_SmallAlpha} provides an asymptotic formula for $F_n(\alpha)$ when $-1+\delta_T<\alpha<1-\frac{\log n}{\log T}-\delta_T$.

For bigger $\alpha$, we prove Conjecture \ref{STPCC} conditionally on the following version of the Hardy-Littlewood conjecture.

Set
\[ \mathfrak{S} := 2 \prod_{p > 2} \left( 1 - \frac{1}{(p-1)^2} \right) \]
and for any integer $h \ne 0$, let 
\[ \mathfrak{S}(h) := \begin{cases}
    0, &\mbox{if } h \text{ is odd}, \\
    \mathfrak{S} \prod_{\substack{ p \, | \, h \\ p > 2 }} \frac{p-1}{p-2}, &\mbox{if } h \text{ is even}.
\end{cases} \]
This function is often referred to as the singular series. Let us further define
\[ \mathfrak{S}_n(h) := \delta((n,h) = 1) \mathfrak{S}(n h). \]

\begin{conjecture}[A uniform version of Hardy-Littlewood]\label{HLC}
    For any $\varepsilon>0$,
    $$\sum_{m\leq x} \Lambda\bigg(\frac{m}{n}\bigg) \Lambda(m \pm h) 
    =  \frac{\mathfrak{S}_n(h)}{n} x + O_\varepsilon(x^{1/2+\varepsilon})$$
    uniformly for $1\leq h,n\leq x^{1-\varepsilon}$.
\end{conjecture}

We recall that, in the classical case, Montgomery \cite{Montgomery1} suggested that assuming the Hardy-Littlewood conjecture one should be able to prove that $F(\alpha)\sim 1$ for $1\leq \alpha\leq 2-\delta$. This was carried out by Bolanz \cite{Bolanz} for $1\leq \alpha\leq\frac{3}{2}-\delta$, and by Goldston and Gonek \cite[Example 4]{GG} in the full range $1\leq \alpha\leq 2-\delta$.
In the following, we seek for an analogue of these results for the twisted pair correlation function $F_n(\alpha)$. Roughly speaking, we consider the range $1-\frac{\log n}{\log T}\leq\alpha\leq 2-C\frac{\log n}{\log T}$ for some positive constant $C$, and we assume Conjecture \ref{HLC}. We show that Conjecture \ref{STPCC} holds for a smoothing of $F_n(\alpha)$. Namely, we define
\begin{equation}\label{SmoothOmega}
    \omega_{\psi_U}(\gamma,\gamma')
    = \frac{2}{\pi} \int_{\mathbb R} \frac{1}{1+(t-\gamma)^2}\frac{1}{1+(t-\gamma')^2} \psi_U\bigg(\frac{t}{T}\bigg) dt.
\end{equation}
Here, $U$ is a real parameter of size $(\log T)^B$ for some $B>0$, and $\psi_U(t)$ is a real-valued weight function such that: 
$\psi_U$ is supported on $[1,2]$,
$\psi_U(t)=1$ if $1+U^{-1}\leq t\leq 2-U^{-1}$,
and $\psi_U^{(j)}(t) \ll U^j$ for all $j\in\mathbb N$.
In these notations, we define
\begin{equation}\notag
    F_n(\alpha;\psi_U) := -\bigg(\frac{T}{2\pi}\frac{\Lambda(n)}{\sqrt{n}} \bigg)^{-1}
    \sum_{T\leq \gamma,\gamma'\leq 2T} n^{i\gamma} T^{i\alpha(\gamma-\gamma')}\omega_{\psi_U}(\gamma,\gamma').
\end{equation}
The effect of the factor $\psi_U$ is that of smoothing the range of summation for the zeros $\gamma$ and $\gamma'$, see e.g. Lemma \ref{IntermediateRangeExplForm}. In practice, this boils down to computing smoothed moments of long Dirichlet polynomials. 
In the setting of \cite[Example 4]{GG}, it is straightforward to pass between the smooth and the sharp cut-off. In our setup for the twisted pair correlation, the transition from $F_n(\alpha)$ to $F_n(\alpha;\psi_U)$ seems to incur an extra factor of size $\sqrt{n}$ in the error term. This appears to cause trouble as soon as $n$ is moderately large.


\begin{theorem}\label{TPC_IntermediateAlpha}
Assume RH and Conjecture \ref{HLC}.
Uniformly for all prime powers $n$ and uniformly for $\alpha$ such that 
\begin{equation}\label{RangeIntermediatealpha}1-\frac{\log n}{\log T} \leq \alpha \leq 2- 48\frac{\log n}{\log T},
\end{equation}
we have
$$F_n(\alpha;\psi_U) = \min\bigg\{ 1 , \frac{\log T}{\Lambda(n)}\bigg(\alpha-1+\frac{\log n}{\log T}\bigg)\bigg\} 
+ O\bigg(\frac{\log\log T}{\Lambda(n)}\bigg).$$
\end{theorem}

We remark that the constant $48$ appearing in Equation \eqref{RangeIntermediatealpha} is not particularly meaningful; we refer the reader to Remark \ref{Remark48} for further details.

The paper is structured as follows. Section \ref{SecPrelim} records various preliminary results. In Sections \ref{SecMainEstimates}-\ref{SecFinalMainEstimates}, we prove Theorem \ref{MainTheorem} assuming the conjectures about the correlation of zeros. Finally, Proposition \ref{STPCCimpliesTPCC}, Theorem \ref{TPC_SmallAlpha}, and Theorem \ref{TPC_IntermediateAlpha} are proven in Sections \ref{SecSTPCCimpTCC}, \ref{SecSmallAlpha}, and \ref{SecIntermAlpha} respectively.

\subsection*{Acknowledgments} 
The authors would like to thank S. Bettin, W. Heap, and J. Keating for several inspiring conversations.
This material is based upon work supported by the Swedish Research Council under grant no. 2021-06594 while the authors were at the Institut Mittag-Leffler during the spring semester of 2024.
The first author is a member of the INdAM group GNAMPA and is supported by the Fonds de recherche du Qu\'ebec - Nature et technologies, Projet de recherche en \'equipe 300951.
The second author was supported by the Swiss National Science Foundation (SNSF) grant no. P400P2$\_$199303. He would like to thank KTH and the University of Oxford for hosting him while this work was created.

\section{Preliminary estimates}\label{SecPrelim}

\subsection{Elementary bounds}

\begin{lemma}\label{CosLogInt}
    For fixed $k \in \mathbb{N}_0$, we have
    \begin{equation}
        \frac{1}{T}\int_T^{2 T} \cos( \rho t) \left(\log \frac{t}{2 \pi} \right)^k dt
        = \begin{cases}
            (\log T)^k + O \left( (\log T)^{k-1} \right), &\mbox{ if } \rho = 0, \\
            O \left( \frac{(\log T)^k}{|\rho| T} \right), &\mbox{ if } \rho \ne 0.
        \end{cases}
    \end{equation}
\end{lemma}

\begin{proof}
    Integration by parts.
\end{proof}

\subsection{Estimates involving zeros of the Riemann zeta function}

We will make use of the following version of Landau's formula due to Gonek \cite{Gonek1}.

\begin{lemma}[Landau-Gonek]\label{GoneksFormula}
    Assume RH and let $2 \le x \le T$. Uniformly for all integers $b < a \le x$, we have
    \begin{equation}
        \frac{1}{T}\sum_{T < \gamma \le 2 T} \left( \frac{a}{b} \right)^{i \gamma} = -\frac{\Lambda(a/b)}{2 \pi} \sqrt{\frac{b}{a}} + O \left( \frac{x (\log T)^2}{T} \right).
    \end{equation}
\end{lemma}

\begin{proof}
    This is an immediate application of \cite[Theorem 1]{Gonek1}, with the adaptation of the Theorem from $0 < \gamma \le T$ to $T < \gamma \le 2T$ going through without difficulty. Denoting, as there, by $\langle \frac{a}{b} \rangle$ the distance of $\frac{a}{b}$ to the nearest prime power other than $\frac{a}{b}$, note that
    \[ \frac{a}{b \langle \frac{a}{b} \rangle} \le a \le x. \]
    Moreover, one has
    \[ \frac{1}{\log \frac{a}{b}} \gg \frac{1}{x}. \]
    This gives the claim.
\end{proof}

We will often make use of the following technical Lemma.

\begin{lemma}\label{6may.2}
    Let $\phi:\mathbb R\to\mathbb R\cup\{\infty\}$ be a function such that $\phi(v)\ll v^{-2}$ for $|v|>1$, and $\phi\in L^p(\mathbb R)$ for all $1\leq p < \infty$. Uniformly for all functions $b:\mathbb R\to \mathbb C$ such that $|b(x)|\leq 1$, denoting $L:=\log T$, we have
    \begin{align}\tag{a}\label{a}
        &\int_T^{2T}b(t) \sum_{\gamma} \phi((t-\gamma)\log x) dt
        = \int_{\mathbb R} b(t) 
        \sum_{T\leq \gamma\leq 2T} \phi((t-\gamma)\log x) dt + O(\sqrt{T}L)\\ \tag{b}\label{b}
        &\int_T^{2T}b(t) \bigg(\sum_{\gamma} \phi((t-\gamma)\log x)\bigg)^2 dt
        = \int_{\mathbb R} b(t) \bigg(\sum_{T\leq \gamma\leq 2T} \phi((\gamma-t)\log x)\bigg)^2dt + O(\sqrt{T}L^2)
        \\ \tag{c}\label{c}
        &\int_T^{2T}  \bigg(\sum_{\gamma} \phi((t-\gamma)\log x)\bigg)^3 dt
        = \int_{\mathbb R} \bigg(\sum_{T\leq \gamma\leq 2T} \phi((\gamma-t)\log x)\bigg)^3dt + O(\sqrt{T}L^3). 
    \end{align}
\end{lemma}

In most of our applications, we will arrive at expressions of the shape on the left-hand side. Restricting the range of summation of zeros and extending the range of integration then allows us to interpret the result as a Fourier transform, especially if we think of the function $b$ as a phase. 

One could certainly generalise the Lemma to any number of zeros without much trouble.

\begin{proof} 
    This is a natural generalization of \cite[Lemma 3.2.1]{LMQ}.
    We only prove \eqref{c}, as the other two can be proven in the same way with easier calculations, and the presence of the 1-bounded function $b$ does not change the argument.
    First we show how to truncate the sums over zeros to the range $[T,2T]$, then we will extend the range of integration. 
    For starters, we write
    \begin{equation}\begin{split}\label{20may.00}
        \bigg(\sum_{\gamma} \phi((\gamma-t)\log x)\bigg)^3
        &= \bigg(\sum_{T\leq \gamma\leq 2T} \phi((\gamma-t)\log x)\bigg)^3 \\
        &\quad +O\bigg(\bigg(\sum_{\substack{\gamma}} \phi((\gamma-t)\log x)\bigg)^2 \sum_{\substack{\gamma''\not\in[T,2T]}} \phi((\gamma''-t)\log x) \bigg).
    \end{split}\end{equation}
    Since $\phi$ is such that $\phi(v)\ll v^{-2}$ for $|v|>1$, we have that (see \cite[Equations (3.6) and (3.7)]{LMQ}), for $t\neq\gamma$
    \begin{equation}\label{20may.01}
        \sum_{\gamma} \phi((\gamma-t)\log x) 
        = \sum_{|\gamma-t| \leq \frac{1}{\log x}} \phi((\gamma-t)\log x) + O(\log (|t|+2))
        =:A+B, 
    \end{equation}
    and
    \begin{equation}\label{20may.02}
        \sum_{\gamma\not\in[T,2T]} \phi((\gamma-t)\log x) 
        = \sum_{\gamma \in I} \phi((\gamma-t)\log x) + O\bigg(\frac{\log T}{2T-t+1}+\frac{\log T}{t-T+1}\bigg)
        :=C+D,
    \end{equation}
    where $I:=\{ \gamma: 2T\leq \gamma\leq 2T+\frac{1}{\log x} \;\text{ or } \; T-\frac{1}{\log x}\leq \gamma\leq T \}$.
    Integrating over $t$ both sides of \eqref{20may.00}, by \eqref{20may.01}, \eqref{20may.02}, and the Cauchy-Schwarz inequality, we get 
    \begin{equation}\begin{split}\label{20may.03}
        \int_T^{2T}\bigg(\sum_{\gamma} \phi((\gamma-t)\log x)\bigg)^3dt
        = &\int_T^{2T}\bigg(\sum_{T\leq \gamma\leq 2T} \phi((\gamma-t)\log x)\bigg)^3 dt \\
        &+O(\mathcal E_{A^2C})+O(\mathcal E_{A^2D})+O(\mathcal E_{B^2C})+O(\mathcal E_{B^2D}).
    \end{split}\end{equation}
    with
    \begin{equation}\begin{split}\notag
        \mathcal E_{B^2D} 
        &\ll \int_T^{2T} (\log t)^2 \bigg(\frac{\log T}{2T-t+1}+\frac{\log T}{t-T+1}\bigg) dt
        \ll (\log T)^4,\\
        \mathcal E_{B^2C} 
        &\ll \int_T^{2T} (\log t)^2 \sum_{\gamma \in I} |\phi((\gamma-t)\log x)| dt
        \ll (\log T)^2\sum_{\gamma\in I}\int_{\mathbb R} |\phi(y)| \frac{dy}{\log x}
        \ll \frac{(\log T)^3}{\log x},\\
        \mathcal E_{A^2D} 
        &\ll  \int_T^{2T} \bigg(\sum_{|\gamma-t| \leq \frac{1}{\log x}} |\phi((\gamma-t)\log x)| \bigg)^2
        \bigg(\frac{\log T}{2T-t+1}+\frac{\log T}{t-T+1}\bigg)
        dt
        =: S_1+S_2+S_3,\\
        \mathcal E_{A^2C}
        &\ll \int_T^{2T} \bigg(\sum_{|\gamma-t| \leq \frac{1}{\log x}} |\phi((\gamma-t)\log x)| \bigg)^2
        \sum_{\gamma''\in I}|\phi((\gamma''-t)\log x)|
        dt
        =: R_1+R_2+R_3.
    \end{split}\end{equation}
    Here $S_1,S_2$ and $S_3$ are defined by splitting the range of integration in the intervals $(T+\sqrt{T},2T-\sqrt{T})$, $(2T-\sqrt{T},2T)$, $(T,T+\sqrt{T})$ respectively, and $R_1,R_2,R_3$ by splitting the integral in the ranges $(T+1,2T-1)$, $(2T-1,2T)$, and $(T,T+1)$ respectively. Let's bound all these contributions:
    \begin{equation}\begin{split}\notag
        S_1 &= \int_{T+\sqrt{T}}^{2T-\sqrt{T}} \bigg(\sum_{|\gamma-t| \leq \frac{1}{\log x}} |\phi((\gamma-t)\log x)| \bigg)^2
        \bigg(\frac{\log T}{2T-t+1}+\frac{\log T}{t-T+1}\bigg)
        dt \\
        &\ll \frac{\log T}{\sqrt{T}} \int_{T+\sqrt{T}}^{2T-\sqrt{T}} \bigg(\sum_{|\gamma-t| \leq \frac{1}{\log x}} |\phi((\gamma-t)\log x)| \bigg)^2
        dt 
        \ll \frac{\log T}{\sqrt{T}} 
        \mathop{\sum\sum}_{\substack{\frac{T}{2}\leq \gamma,\gamma' \leq \frac{5T}{2}: \\ |\gamma-\gamma'| \leq 1}} 1
        \ll \sqrt{T} (\log T)^3
    \end{split}\end{equation}
    since $\phi\in L^2(\mathbb R)$ and $\sum_{|\gamma-v|\leq 1} 1 \ll \log(|v|+2)$. Similarly,
    \begin{equation}\begin{split}\notag
        S_2 &= \int_{2T-\sqrt{T}}^{2T}
        \bigg(\sum_{|\gamma-t| \leq \frac{1}{\log x}} |\phi((\gamma-t)\log x)| \bigg)^2
        \bigg(\frac{\log T}{2T-t+1}+\frac{\log T}{t-T+1}\bigg)
        dt \\
        &\ll \log T \mathop{\sum\sum}\limits_{\substack{2T-\sqrt{T}-1\leq \gamma,\gamma' \leq 2T+1 \\ |\gamma-\gamma'|\leq 1 }} 1
        \ll (\log T)^2 \sum_{2T-\sqrt{T}-1\leq \gamma \leq 2T+1}1
        \ll \sqrt{T}(\log T)^3
    \end{split}\end{equation}
    and
        \begin{equation}\begin{split}\notag
        S_3 &= \int_{T}^{T+\sqrt{T}}
        \bigg(\sum_{|\gamma-t| \leq \frac{1}{\log x}} |\phi((\gamma-t)\log x)| \bigg)^2
        \bigg(\frac{\log T}{2T-t+1}+\frac{\log T}{t-T+1}\bigg)
        dt 
        \ll \sqrt{T}(\log T)^3.
    \end{split}\end{equation}
    Therefore,
    \begin{equation}\label{6may.10}
        \mathcal E_{A^2D}
        \ll \sqrt{T}(\log T)^3.
    \end{equation}
    Finally, we need to bound $\mathcal E_{A^2C}$, i.e. $R_1,R_2$ and $R_3$. For $t\in(T+1,2T-1)$ and $\gamma''\in (2T,2T+\frac{1}{\log x})$, we have $\phi((\gamma''-t)\log x)\ll \frac{1}{2T-t+1}$. A similar consideration for the case $\gamma''\in(T-\frac{1}{\log x},T)$ leads to
    \begin{equation}\begin{split}\notag
        R_1 = \int_{T+1}^{2T-1} \bigg(\sum_{|\gamma-t| \leq \frac{1}{\log x}} |\phi((\gamma-t)\log x)| \bigg)^2
        \sum_{\gamma''\in I}|\phi((\gamma''-t)\log x)|
        dt
        \ll \mathcal E_{A^2D}
        \ll \sqrt{T}(\log T)^3
    \end{split}\end{equation}
    by \eqref{6may.10}.
    We now turn to $R_2$:
    \begin{equation}\notag
        R_2 
        = \int_{2T-1}^{2T} \bigg(\sum_{|\gamma-t| \leq \frac{1}{\log x}} |\phi((\gamma-t)\log x)| \bigg)^2
        \sum_{2T\leq \gamma''\leq 2T+\frac{1}{\log x}}|\phi((\gamma''-t)\log x)|
        dt + O\bigg(\frac{(\log T)^3}{T^2}\bigg)
    \end{equation}
    since the contribution from the case $T-\frac{1}{\log x}\leq \gamma''\leq T$ is $\ll (\log T)^3/T^2$. Then the above is bounded by
    \begin{equation}\begin{split}\notag
        &\ll \mathop{\sum\sum\sum}\limits_{2T-2\leq \gamma,\gamma',\gamma''\leq 2T+1} \int_{2T-1}^{2T} |\phi((\gamma-t)\log x)||\phi((\gamma'-t)\log x)|
        |\phi((\gamma''-t)\log x)| dt
        \ll (\log T)^3
    \end{split}\end{equation}
    since $\phi\in L^3(\mathbb R)$. The same argument gives $R_3\ll (\log T)^3$ and therefore
    \begin{equation}\notag
        \mathcal E_{A^2C} \ll \sqrt{T}(\log T)^3.
    \end{equation}
    All these considerations show that the error terms in \eqref{20may.03} are $\ll \sqrt{T}(\log T)^3$.
    
    To extend the range of integration on the left hand-side of Equation \eqref{20may.03}, it suffices to bound the integral over $(2T,+\infty)$; the other range $(-\infty,T)$ is analogous. If $t>2T+1$, then $\sum_{T\leq \gamma\leq 2T} |\phi((\gamma-t)\log x)|\ll \log T/(t-2T+1)$ and therefore
    $$\int_{2T+1}^{+\infty} \bigg(\sum_{T\leq \gamma\leq 2T} h((\gamma-t)\log x)\bigg)^3 dt \ll (\log T)^3 .$$
    For $2T\leq t\leq 2T+1$, applying \eqref{20may.01} we get
    \begin{equation}\begin{split}\notag
        &\int_{2T}^{2T+1}  \bigg(\sum_{T\leq \gamma\leq 2T} \phi((\gamma-t)\log x)\bigg)^3 dt
        \ll \int_{2T}^{2T+1}  \bigg(\sum_{|\gamma-t| \leq \frac{1}{\log x}} |\phi((\gamma-t)\log x)| \bigg)^3 dt + (\log T)^3\\
        &\ll \mathop{\sum\sum\sum}_{2T-1\leq \gamma,\gamma',\gamma''\leq 2T+2} \int_{2T}^{2T+1}  |\phi((\gamma-t)\log x)||\phi((\gamma'-t)\log x)||\phi((\gamma''-t)\log x)| dt + (\log T)^3
    \end{split}\end{equation}
    that is $\ll(\log T)^3$ since $ \phi\in L^3(\mathbb R)$.
\end{proof}

\subsection{Approximation of the third moment in terms of sums of primes and zeros}

We borrow from \cite{LMQ} the following auxiliary functions:
\begin{equation}\begin{split}\label{DefAuxFunctions}
&f(u) := u\int_0^{\infty} \frac{\sinh(y(1-u))}{\cosh y} dy, \quad \hspace{.1cm} \text{for } u\in(0,2)\\
&g(u) := \int_0^{\infty} \frac{e^{-y}\cosh(uy)}{\cosh y} dy, \quad \hspace{.8cm} \text{for } u\in(-2,2)\\
&h(u) := \cos u\int_0^{\infty} \frac{y}{\cosh y} \frac{dy}{y^2+u^2}, \quad \text{for } u\in\mathbb R\setminus \{0\}.
\end{split}\end{equation}
The main properties of $f,g,h$ are summarized in \cite[Lemma 2.1.1]{LMQ}. In particular, we recall that $g$ and $h$ are even functions such that
\begin{equation}\label{PropertiesAuxiliaryFunctions}
g(u) = \frac{1-f(u)}{u} \quad \text{ and } \quad
\hat h(a) = \begin{cases} \pi g(2\pi a) & \text{if } |a|\leq \frac{1}{2\pi} \\ \frac{1}{2|a|} & \text{if } |a|> \frac{1}{2\pi} .\end{cases}
\end{equation}

\begin{proposition}\label{P&Z}
    Define
    \[ P(t) := \sum_{n \le x} \frac{\Lambda(n) \cos(t \log n)}{\sqrt{n} \log n} f \left( \frac{\log n}{\log x} \right)\] 
    and 
    \[ Z(t) := - \sum_{\gamma} h((\gamma - t) \log x) + \frac{\hat h(0) \log \frac{t}{2 \pi}}{2 \pi \log x}. \]
    Assuming RH, for any $2 \le x \le T$ we have
    \begin{equation}\notag
        M_3^{\Re}(T) = \frac{1}{T} \int_T^{2 T} \left( \Re \log \zeta(1/2+it) \right)^3 \, dt = \frac{1}{T} \int_T^{2 T} (P(t) + Z(t))^3 \, dt + O \left( \frac{\sqrt{x} \log \log T}{T \log^2 x} \right).
    \end{equation}
\end{proposition}

\begin{proof}
    By \cite[Lemma 2.2.1]{LMQ}, assuming RH we have
    \[ \Re \log \zeta(1/2+it) = P(t) + Z(t) + O \left( \frac{\sqrt{x}}{t \log^2 x} \right). \]
    Hence, we obtain (bringing the $O$ factor to the left-hand side and expanding) that
    \begin{equation}\begin{split}\label{6may.1}
        M_3^{\Re}(T) &= \frac{1}{T} \int_T^{2 T} (P(t) + Z(t))^3 \, dt + O \left( \frac{\sqrt{x}}{T^2 \log^2 x} \int_T^{2 T} (\Re \log \zeta(1/2+it))^2 \, dt \right) \\
        &+ O \left( \frac{x}{T^3 \log^4 x} \int_T^{2 T} \left| \Re \log \zeta(1/2+it) \right| \, dt \right) + O \left( \frac{x^{3/2}}{T^3 \log^6 x} \right).
    \end{split}\end{equation}
    Next, by the Selberg Central Limit Theorem \cite{SelbergCLT,SelbergOandN}, we have
    \[ \int_{T}^{2 T} (\Re \log \zeta(1/2+it))^2 \, dt \ll T \log \log T, \]
    which also implies, using Cauchy-Schwarz inequality, that
    \[ \int_T^{2 T} \left| \Re \log \zeta(1/2+it) \right| \, dt \ll T (\log \log T)^{1/2}. \]
    The claim follows upon plugging these bounds into \eqref{6may.1} and noting that $2 \le x \le T$.
\end{proof}

Similarly, one can approximate the imaginary part of log-zeta. The auxiliary functions are:
\begin{equation}\begin{split}\label{DefAuxFunctionsIm}
&\mathfrak f(u) := \frac{\pi}{2}\cot\bigg(\frac{\pi}{2}u\bigg), \quad \hspace{2cm} \text{for } u\in(0,2)\\
&\mathfrak h(u) := \sin u\int_0^{\infty} \frac{y}{\sinh y} \frac{dy}{y^2+u^2}, \quad \text{for } u\in\mathbb R\setminus \{0\}.
\end{split}\end{equation}
We note that $\mathfrak f$ is an even function such that $\mathfrak f(u)=1+O(u^2)$ as $u\to0$, and $\mathfrak h$ is an odd function.

\begin{proposition}\label{P&ZIm}
    Define
    \[ \mathcal P(t) := -\sum_{n \le x} \frac{\Lambda(n) \sin(t \log n)}{\sqrt{n} \log n} \mathfrak f \left( \frac{\log n}{\log x} \right)\] 
    and 
    \[ \mathcal Z(t) :=  \sum_{\gamma} \mathfrak h((t-\gamma) \log x) . \]
    Assuming RH, for any $2 \le x \le T$ we have
    \begin{equation}\notag
        M_3^{\Im}(T) = \frac{1}{T} \int_T^{2 T} \left( \Im \log \zeta(1/2+it) \right)^3 \, dt = \frac{1}{T} \int_T^{2 T} (\mathcal P(t) + \mathcal Z(t))^3 \, dt + O \left( \frac{\sqrt{x} \log \log T}{T \log^2 x} \right).
    \end{equation}
\end{proposition}

\begin{proof}
    Assume RH and let $2\leq x\leq T$. Work of Goldston (see \cite{Goldston1}, Equation (2.9) and (2.12)) implies
    \[ \Im \log \zeta(1/2+it) 
    = \pi S(t)
    = \mathcal P(t) + \mathcal Z(t) + O \left( \frac{\sqrt{x}}{t \log^2 x} \right). \]
    The claim follows from the same argument as in the last proof, by an application of Selberg's central limit theorem.
\end{proof}

\subsection{Estimates for singular series averages}

We will later require estimates for averages of the singular series $\mathfrak{S}_n$, which we collect below. None of these estimates are particularly novel, and most of them are not stated in the strongest possible way.

For simplicity we will assume in the proofs throughout this section that $n = q^a$ is the power of an odd prime $q > 2$. The case $q = 2$ can be obtained with little modification.

\begin{lemma}\label{SigmaNCM}
    Assume RH. For any $0 < \varepsilon < \frac{1}{2}$, we have
    \begin{equation}
        S_n(y) := \sum_{h \le y} (y-h) \mathfrak{S}_n(h) = 
        \begin{cases}
        \frac{y^2}{2} - \frac{\log y}{2} y + A y + O \left( \frac{y^2}{q} \right) + O_\varepsilon \left( y^{1/2+\varepsilon} \right), &\mbox{ if } y < q, \\
        \frac{y^2}{2} - \frac{\Lambda(n)}{2} y + O \left( y  \left( \frac{q}{y} \right)^{1/2-\varepsilon} \right), &\mbox{ if } y \ge q,    
        \end{cases}
    \end{equation}
    where
    \[ A = \frac{1-\gamma-\log 2 \pi}{2}. \]
\end{lemma}

\begin{proof}
    For $q > y$, we have $(h,n) = 1$ for all $h \le y$ and hence $\mathfrak{S}_n(h) = \frac{q-1}{q-2} \mathfrak{S}(h)$. The claim then follows immediately from \cite[Proposition 2]{FG1}, which states that
    \[ \sum_{h \le y} (y-h) \mathfrak{S}(h) = \frac{y^2}{2} - \frac{\log y}{2} y + O(y^{1/2+\varepsilon}). \]

    So suppose now that $q \le y$.
    The claim is a consequence of fairly standard contour integration arguments, and as such, we will confine ourselves to a sketch of proof following the lines of \cite{GoldstonSuriajaya1,GoldstonSuriajaya2,Rodgers1}.
    
    Note that we may as well restrict to the case when $n = q$ is a prime, since all the quantities involved depend only on $q$.

    For $\Re s > 1$, let
    \begin{equation}
        \mathcal F_n(s) := \sum_{h \ge 1} \frac{\mathfrak{S}_n(h)}{h^s}
    \end{equation}
    be the Dirichlet series associated to $\mathfrak{S}_n$. Then we have \cite[(6.5)]{Rodgers1}
    \begin{equation}
        \mathcal F_n(s) =  \left( 1 - \frac{1}{2^{s+1}} \right)\mathfrak{S} A_q(s)  \zeta(s) \zeta(s+1) \mathcal{G}(s) 
    \end{equation}
    with
    \begin{equation}
        \mathcal G(s) := \prod_{p > 2} \left( 1 + \frac{2}{(p-2) p^{s+1}} - \frac{1}{(p-2) p^{2 s + 1}} \right)
    \end{equation}
    and
    \begin{equation}
        A_q(s) := \left( 1 - \frac{1}{q-1} + \frac{1}{q^s - 1} \right)^{-1}.
    \end{equation}
    This provides a meromorphic continuation of $\mathcal F_n$ to $\Re s > - \frac{1}{2}$. Note that since we are assuming the Riemann Hypothesis, one could continue meromorphically to $\Re s > - \frac{3}{4}$ (by pulling out $\zeta(2s+2)^{-1}$ essentially), but this is not necessary for our purposes and would cause small complications later.

    Perron's formula for the Ces\`aro mean gives
    \begin{equation}\label{17may.1}
        S_n(y) = \frac{1}{2 \pi i} \int_{(2)} \mathcal{F}_n(s) y^{s+1} \, \frac{ds}{s(s+1)} =: \frac{1}{2 \pi i} \int_{(2)} \tilde{\mathcal{F}}_n(s) \, ds.
    \end{equation}

    One verifies that the only poles of $\tilde{\mathcal{F}}_n$ in the region $\Re s > - \frac{1}{2}$ are simple poles located at $s = 1$ and $s = 0$. At $s = 0$, this is because the double pole coming from $\frac{\zeta(s+1)}{s}$ is compensated by a simple zero of the factor $A_q(s)$. Consequently it is somewhat different from the situation in \cite{GoldstonSuriajaya2} where  $A_q(s)$ is not present, resulting in a double pole at $s = 0$ that ultimately gives an additional term of order $y \log y$.

    Next, note that
    \begin{equation}
        \mathrm{Res}_{s = 1} \tilde{\mathcal{F}}_n(s) = \frac{y^2}{2}
    \end{equation}
    and
    \begin{equation}
        \mathrm{Res}_{s = 0} \tilde{\mathcal{F}}_n(s) = - \frac{\Lambda(n)}{2}y.
    \end{equation}

    Cutting off the integral in \eqref{17may.1} at height $T$ for some large parameter $T$, we obtain
    \begin{equation}\label{17may.2}
        S_{n}(y) = \frac{1}{2 \pi i} \int_{2 - i T}^{2 + i T} \tilde{\mathcal{F}}_n(s) \, ds + O \left( \frac{y^3}{T} \right)
    \end{equation}
    since $\mathcal F_n(2 + i t) \ll 1$ uniformly in $n$ and $t$. 
    
    Now we want to shift the contour to $\Re s = \sigma_\varepsilon := - \frac{1}{2} + \varepsilon$. Using the residue theorem, we deduce
    \begin{equation}\label{17may.3}
        S_{n}(y) = \frac{y^2}{2} - \frac{\Lambda(n)}{2} y + \frac{1}{2 \pi i} \left( \int_{2 - i T}^{\sigma_\varepsilon - i T} + \int_{\sigma_\varepsilon - i T}^{\sigma_\varepsilon + i T} + \int_{\sigma_\varepsilon + i T}^{2 + i T} +  \right)\tilde{\mathcal{F}}_n(s) \, ds + O \left( \frac{y^3}{T} \right).
    \end{equation}
    In order to bound these three integrals, recall (see e.g. \cite[(5.1.6)]{Titchmarsh1}) that under RH we have $\zeta(\sigma+i t) \ll (|t|+1)^{\mu(\sigma)+\varepsilon/4}$ with
    \begin{equation}
        \mu(\sigma) = 
        \begin{cases}
            \frac{1}{2} - \sigma, &\mbox{if } \sigma \le \frac{1}{2}, \\
            0, &\mbox{if } \sigma > \frac{1}{2}
        \end{cases}
    \end{equation}
    as long as (say) $|\sigma + i t - 1| \ge 1$. Moreover, one verifies that
    \begin{equation}
        A_q(\sigma + i t) \ll
        \begin{cases}
            1, &\mbox{if } \sigma \ge 0, \\
            q^{-\sigma}, &\mbox{if } \sigma < 0
        \end{cases}
    \end{equation}
    and that
    \[ \mathcal{G}(\sigma + i t) \asymp 1 \]
    uniformly for $\sigma \ge \sigma_\varepsilon$.

    For the horizontal parts, we can simply bound
    \begin{equation}
        \tilde{\mathcal{F}}_n(\sigma \pm iT) \ll \frac{y^3 \sqrt{q}}{T}
    \end{equation}
    uniformly for $\sigma_\varepsilon \le \sigma \le 2$.
    On the vertical parts, we need to be slightly more precise, namely we estimate
    \begin{equation}
        \tilde{\mathcal{F}}_n(\sigma_\varepsilon + i t) \ll_\varepsilon \frac{y^{1+\sigma_\varepsilon}}{(|t|+1)^{2}} (|t|+1)^{1-\varepsilon +2 \varepsilon/4}q^{-\sigma_\varepsilon} = y \left(\frac{q}{y}\right)^{1/2-\varepsilon} (|t|+1)^{-1-\varepsilon/2}.
    \end{equation}
    Plugging these bounds into \eqref{17may.3} and integrating, we arrive at
    \begin{equation}
        S_{n}(y) = \frac{y^2}{2} - \frac{\Lambda(n)}{2} y + O \left( y \left( \frac{q}{y} \right)^{1/2-\varepsilon} \right) + O \left( \frac{y^3 \sqrt{q}}{T} \right).
    \end{equation}
    Taking say $T = y^3 \sqrt{q}$ gives the claim.
\end{proof}

We will also make use of the following auxiliary functions. Roughly following notation of \cite{Chan1}, for $\alpha \ge 0$, set
\begin{equation}
    S_{\alpha,n}(y) := \sum_{h \le y} \mathfrak{S}_n(h) h^\alpha - \frac{y^{\alpha + 1}}{\alpha + 1}.
\end{equation}
Moreover, for  $\alpha > 1$ define
\begin{equation}
    T_{\alpha,n}(y) := \sum_{h > y} \frac{\mathfrak{S}_n(h)}{h^\alpha} - \frac{1}{(\alpha -1)y^{\alpha - 1}}.
\end{equation}

\begin{lemma}\label{LemmaAsS0}
    Uniformly in $y$ and $n$, we have
    \begin{equation}\label{EqAsS0}
        S_{0,n}(y) = - \frac{\min \{ \Lambda(n), \log y \} }{2} + O((\log y)^{2/3}).
    \end{equation}
    In particular, for (fixed) $\alpha > 0$ we have
    \begin{equation}\begin{split}
        S_{\alpha,n}(y) = O_\alpha\left(y^{\alpha} (\log y)^{2/3}\right)
    \end{split}\end{equation}
    and for (fixed) $\alpha > 1$
    \begin{equation}\begin{split}
        T_{\alpha,n}(y) = O_\alpha\left(\frac{(\log y)^{2/3}}{y^\alpha}\right).
    \end{split}\end{equation}
    
\end{lemma}

\begin{proof}
    Equation \eqref{EqAsS0} will follow fairly quickly from \cite[Proposition 1]{FG1}. The other two claims are immediate consequences by means of summation by parts.

    Following the notation of \cite{FG1}, set
    \[ H_q(h) := \prod_{\substack{p \, | \, h \\ (p,2q) = 1}} \frac{p-1}{p-2}. \]
    We then have
    \begin{align*}
        S_{0,n}(2y) &= \frac{q-1}{q-2} \sum_{\substack{h \le 2y \\ (h,q) = 1}} \mathfrak{S}(h) - 2y = \frac{q-1}{q-2} \mathfrak{S} \left( \sum_{h \le y} H_q(h) - \sum_{h' \le y/q} H_q(h') \right) - 2y.
    \end{align*} 
    Note that the second sum is empty if $q > y$.
    The statement then follows upon employing \cite[(2.2)]{FG1} after a short calculation.
\end{proof}

\begin{lemma}\label{LemmaAsForf_n}
    Let
    \begin{equation}
        f_n(y) := y T_{2,n}(y) + \frac{S_{2,n}(y)}{y^3}.
    \end{equation}
    We have
    \begin{equation}
        f_n(y) = \begin{cases}
            -\frac{1}{2 y} + O \left( y^{-5/4} \right) + O \left( q^{-1} \right), &\mbox{ if } y < q, \\[4pt]
            O \left( \frac{q^{1/4}}{y^{5/4}} \right), &\mbox{ if } y \ge q.
        \end{cases}
    \end{equation}
\end{lemma}

\begin{proof}
    Integration by parts implies
    \[ \frac{S_{2,n}(y)}{y^3} = \frac{S_{0,n}(y)}{y} - \frac{2}{y^3} \int_1^y S_{0,n}(t) t \, dt + O(y^{-3}) \]
    and similarly
    \[ y T_{2,n}(y) = - \frac{S_{0,n}(y)}{y} + 2 y \int_y^\infty \frac{S_{0,n}(t)}{t^3} \, dt. \]
   Inspired by Lemma \ref{LemmaAsS0}, we define $\varepsilon_n$ via the equality
    \[ S_{0,n}(y) = - \frac{1}{2} \min \{ \Lambda(n), \log y \} + \varepsilon_n(y). \]
    Plugging this expression in and integrating, we deduce that
    \begin{equation*}
        f_n(y) = 2 y \int_y^\infty \frac{\varepsilon_n(t)}{t^3} \, dt - \frac{2}{y^3} \int_1^y \varepsilon_n(t) t \, dt + \begin{cases}
            -\frac{1}{2 y} + O \left( \frac{y}{q^2} \right) + O \left( \frac{1}{y^3} \right), &\mbox{if } y < q, \\
            O \left( \frac{q^2}{y^3} \right), &\mbox{if } y \ge q.
        \end{cases}
    \end{equation*}
    The claim will thus follow once we can show that
    \begin{equation}\label{26sep.1}
        2 y \int_y^\infty \frac{\varepsilon_n(t)}{t^3} \, dt - \frac{2}{y^3} \int_1^y \varepsilon_n(t) t \, dt = \begin{cases}
            O \left( \frac{1}{q} \right) + O \left( \frac{1}{y^{5/4}} \right), &\mbox{if } y < q, \\
            O \left( \frac{q^{1/4}}{y^{5/4}} \right), &\mbox{if } y \ge q.
            \end{cases}
    \end{equation}
    In order to achieve this we begin by noting that, by application of Lemma \ref{SigmaNCM}, we have
    \begin{equation}\begin{split}\label{26sep.2}
        \mathcal{E}_n(y) &:= \int_1^y \varepsilon_n(t) \, dt
        = \int_1^y \left( \sum_{h \le t} \mathfrak{S}_n(h) - t + \frac{1}{2} \min \{ \Lambda(n), \log t \} \right) \, dt \\
        &= \sum_{h \le y} (y-h) \mathfrak{S}_n(h) - \frac{y^2}{2} + \frac{1}{2} \left( y \min \{ \Lambda(n), \log y \} - \min \{q,y \} \right) + O(1) \\
        &= \begin{cases}
            \left( A - \frac{1}{2} \right) y + O \left( \frac{y^2}{q} \right) + O(y^{3/4}), &\mbox{if } y < q, \\
            O \left( y^{3/4} q^{1/4} \right), &\mbox{if } y \ge q.
        \end{cases}
    \end{split}\end{equation}
    Another integration by part gives
    \begin{equation}
        2 y \int_y^\infty \frac{\varepsilon_n(t)}{t^3} \, dt - \frac{2}{y^3} \int_1^y \varepsilon_n(t) t \, dt
        = \frac{-4 \mathcal{E}_n(y)}{y^2} + 6 y \int_y^\infty \frac{\mathcal{E}_n(t)}{t^4} \, dt + \frac{2}{y^3} \int_1^y \mathcal{E}_n(t) \, dt.
    \end{equation}
    Inserting \eqref{26sep.2} implies \eqref{26sep.1} after a routine calculation.
\end{proof}

\section{The contribution from three prime powers}\label{SecMainEstimates}

We begin by estimating the contribution to $M_3^{\Re}(T)$ in Proposition \ref{P&Z} that comes from integrating $P(t)^3$. This case turns out to be the most straight-forward, and gives rise to the prime constant $c_P$ in \eqref{TheoremRP}. We have not attempted to maximise the range of $x$, since the subsequent sections will pose more strict requirements.

\begin{proposition}
    For any $2 \le x \le T^{1/3}$, we have
    \begin{equation}\notag
        \frac{1}{T} \int_T^{2 T} P(t)^3 \, dt = c_P + O \left( \frac{1}{\log x} \right).
    \end{equation}
\end{proposition}

\begin{proof}
    Note that
    \begin{equation}\begin{split} \label{2may.1}
        \frac{1}{T}\int_{T}^{2T} P(t)^3 dt
        &= \frac{1}{8}\sum_{a,b,c\leq x} \frac{\Lambda(a)\Lambda(b)\Lambda(c)}{\sqrt{abc}\log a\log b\log c} f\bigg(\frac{\log a}{\log x}\bigg)f\bigg(\frac{\log b}{\log x}\bigg)f\bigg(\frac{\log c}{\log x}\bigg)\\
        & \hspace{3.2cm} \times \frac{1}{T}\int_T^{2T} (a^{it}+a^{-it})(b^{it}+b^{-it})(c^{it}+c^{-it})dt.
\end{split}\end{equation}
From the fact that the length of the respective sums is rather short, we might expect the only contribution to come from the ``diagonal'' terms where $ab = c$ or a permutation thereof, and that is indeed what we will see now. 

We will deal as an example with the case where the integrand is $\left(\frac{ab}{c}\right)^{it}$. There are $5$ further cases where the integrand is of essentially the same shape, and it is straight-forward to see that they give the same contribution. We will then explain how to bound the part with integrand $(abc)^{it}$ and remark that $(abc)^{-it}$ can be dealt with in the same way (and is in fact simply the conjugate).

For $a, b, c \le x$ with $ab \ne c$, we can use the rather crude bound
\[ \int_{T}^{2 T} \left( \frac{ab}{c} \right)^{it} dt \ll \frac{1}{\left|\log \frac{ab}{c}\right|} \ll x \]
to deduce that
\begin{align*}
    &\frac{1}{8}\sum_{a,b,c\leq x} \frac{\Lambda(a)\Lambda(b)\Lambda(c)}{\sqrt{abc}\log a\log b\log c} f\bigg(\frac{\log a}{\log x}\bigg)f\bigg(\frac{\log b}{\log x}\bigg)f\bigg(\frac{\log c}{\log x}\bigg)\frac{1}{T}\int_T^{2T} \left( \frac{ab}{c} \right)^{it} dt \\
    = &\frac{1}{8} \sum_{ab \le x} \frac{\Lambda(a) \Lambda(b) \Lambda(ab)}{a b \log a \log b \log (ab)}  f\bigg(\frac{\log a}{\log x}\bigg)f\bigg(\frac{\log b}{\log x}\bigg)f\bigg(\frac{\log a b}{\log x}\bigg) + O \left( \frac{x}{T} \sum_{a, b, c \le x} \frac{\Lambda(a)\Lambda(b)\Lambda(c)}{\sqrt{abc}\log a\log b\log c} \right) \\
    = &\frac{1}{8} \sum_{ab \le x} \frac{\Lambda(a) \Lambda(b) \Lambda(ab)}{a b \log a \log b \log (ab)}  f\bigg(\frac{\log a}{\log x}\bigg)f\bigg(\frac{\log b}{\log x}\bigg)f\bigg(\frac{\log a b}{\log x}\bigg) + O \left( \frac{x^{5/2}}{T} \right).
\end{align*}
Here, we used the fact that $f$ is bounded, and in the last step we simply bounded $\Lambda(n) \le \log n$. In the main term, note that the appearance of the factor $\Lambda(ab)$ means that we can assume $a = p^\alpha$ and $b = p^\beta$ to be powers of the same prime $p$. Thus, we can write the main term as
\begin{align*} &= \frac{1}{8} \sum_{p, \alpha, \beta : p^{\alpha+\beta} \le x} \frac{1}{\alpha \beta (\alpha+\beta) p^{\alpha+\beta}} f\bigg(\frac{\alpha \log p}{\log x}\bigg)f\bigg(\frac{\beta \log p}{\log x}\bigg)f\bigg(\frac{(\alpha + \beta) \log p}{\log x}\bigg) \\
&= \frac{1}{8} \sum_{p, \gamma : p^\gamma \le x} \frac{1}{\gamma p^\gamma} f\bigg(\frac{\gamma \log p}{\log x} \bigg) \sum_{\alpha + \beta = \gamma} \frac{1}{\alpha \beta} f\bigg(\frac{\alpha \log p}{\log x}\bigg)f\bigg(\frac{\beta \log p}{\log x}\bigg).
\end{align*}
Simply writing $f(x) = 1 + O(x)$, one verifies that this equates to
\[ = \frac{1}{8} \sum_{p, \gamma \ge 2} \frac{1}{\gamma p^\gamma} \sum_{\alpha + \beta = \gamma} \frac{1}{\alpha \beta} + O \left( \frac{1}{\log x} \right). \]
Summarising the estimates so far, we have seen that the six contributions to \eqref{2may.1} coming from the six terms involving the integrand $\left( \frac{ab}{c} \right)^{i t}$ together with its permutations and conjugates combine to
\[ = c_P + O \left( \frac{1}{\log x} \right). \]
It remains to bound the contribution coming from $(abc)^{it}$ (and its conjugate), which is even more direct. We can simply note that
\[ \int_T^{2T} (abc)^{i t} \, dt \ll 1 \]
for any prime powers $a, b$ and $c$, so that
\begin{align*}
    &\sum_{a,b,c\leq x} \frac{\Lambda(a)\Lambda(b)\Lambda(c)}{\sqrt{abc}\log a\log b\log c} f\bigg(\frac{\log a}{\log x}\bigg)f\bigg(\frac{\log b}{\log x}\bigg)f\bigg(\frac{\log c}{\log x}\bigg)\frac{1}{T}\int_T^{2T} (abc)^{it} dt \\
    & \hspace{9cm} \ll \frac{1}{T} \left( \sum_{a \le x} \frac{1}{\sqrt{a}} \right)^3 \ll \frac{x^{3/2}}{T}.
\end{align*}
The claim follows.
\end{proof}

\begin{proposition}
    For any $2 \le x \le T^{1/3}$, we have
    \begin{equation}\notag
        \frac{1}{T} \int_T^{2 T} \mathcal P(t)^3 \, dt = O \left( \frac{1}{\log x} \right).
    \end{equation}
\end{proposition}

\begin{proof}
    We will not give the proof details, which are a straight-forward adaptation of the previous argument. We simply note that due to the appearance of sine in place of cosine, the six terms that contribute will consist of three each with opposite signs, thus cancelling in the end.
\end{proof}

\section{The contribution from two prime powers and a zero}

\begin{proposition}\label{P^2Z_Re}
    Assume RH. Uniformly for $2 \le x \le T^{1/4}$, we have
    \begin{equation}\notag
        \frac{1}{T} \int_T^{2 T} P(t)^2 Z(t) \, dt = O \left( \frac{1}{\log x} \right).
    \end{equation}
\end{proposition}

\begin{proof}
    Write
    \[ \frac{1}{T} \int_T^{2 T} P(t)^2 Z(t) \, dt = \mathcal{I}_{P^2 \hat h} - \mathcal{I}_{P^2 \gamma}  \]
    with
    \begin{equation}\begin{split}
        \mathcal{I}_{P^2 \hat h} &= \frac{\hat h(0)}{2 \pi \log x} \sum_{a, b \le x} \frac{\Lambda(a) \Lambda(b)}{\sqrt{ab} (\log a) (\log b)} f \left( \frac{\log a}{\log x} \right) f \left( \frac{\log b}{\log x} \right) \times \\ &\hspace{6cm}\frac{1}{T} \int_T^{2 T} \cos(t \log a) \cos(t \log b) \log \frac{t}{2 \pi} \, dt
    \end{split}\end{equation}
    and
    \begin{equation}\begin{split}\label{PsqGamma}
        \mathcal{I}_{P^2 \gamma} &= \sum_{a, b \le x} \frac{\Lambda(a) \Lambda(b)}{\sqrt{ab} (\log a) (\log b)} f \left( \frac{\log a}{\log x} \right) f \left( \frac{\log b}{\log x} \right) \times \\ &\hspace{5cm}\sum_\gamma \frac{1}{T} \int_T^{2 T} \cos(t \log a) \cos(t \log b) h((\gamma-t)\log x) \, dt.
    \end{split}\end{equation}
    By Lemma \ref{CosLogInt} combined with standard trigonometric identities, we have
    \[ \frac{1}{T} \int_T^{2 T} \cos(t \log a) \cos(t \log b) \log \frac{t}{2 \pi} \, dt = \begin{cases}
        \frac{1}{2} (\log T + C) + O\left(\frac{\log T}{T} \right), &\mbox{ if } a = b, \\
        O \left( \frac{x \log T}{T} \right), &\mbox{ if } a \ne b
    \end{cases} \]
    with $C = \log 2/\pi - 1$.
    Therefore, we see that
    \begin{equation}\begin{split}
        \mathcal{I}_{P^2 \hat h} &= \left(\log T + C + O \left( \frac{\log T}{T} \right)\right) \frac{\hat h (0)}{4 \pi \log x} \sum_{n \le x} \frac{\Lambda(n)^2}{n (\log n)^2} f^2 \left( \frac{\log n}{\log x} \right) \\ &\hspace{3cm} + O \left( \frac{x \log T}{T \log x} \sum_{a, b \le x} \frac{\Lambda(a) \Lambda(b)}{\sqrt{ab} (\log a) (\log b)} f \left( \frac{\log a}{\log x} \right) f \left( \frac{\log b}{\log x} \right) \right).
    \end{split}\end{equation}
    Following the same strategy as before, one obtains that the last error term here is $O \left( \frac{x^2 \log T}{T \log x} \right)$,
    which is certainly acceptable. Furthermore, we have
    \[ \sum_{n\le x} \frac{\Lambda(n)^2}{n (\log n)^2}f^2 \left( \frac{\log n}{\log x} \right) \ll \sum_{n\le x} \frac{\Lambda(n)^2}{n (\log n)^2} \ll \log \log x,\]
    so that
    \begin{equation}\label{Psqhath}
        \mathcal{I}_{P^2 \hat h} = \left(\log T + C \right) \frac{\hat h (0)}{4 \pi \log x} \sum_{n \le x} \frac{\Lambda(n)^2}{n (\log n)^2} f^2 \left( \frac{\log n}{\log x} \right) + O \left( \frac{x^2 \log T}{T} \right).
    \end{equation}
    We now turn our attention to the right-hand side of \eqref{PsqGamma}. First of all, by Lemma \ref{6may.2} \eqref{a}, we have
    \begin{equation}\begin{split}
        &\sum_\gamma \int_T^{2 T} \cos(t \log a) \cos(t \log b) h((\gamma-t)\log x) \, dt \\ &\hspace{2.5cm} = \sum_{T < \gamma \le 2 T} \int_{\mathbb{R}} \cos(t \log a) \cos(t \log b) h((\gamma-t)\log x) \, dt + O \left( \sqrt{T} \log T \right).
    \end{split}\end{equation}
    Writing $\cos(t \log a) \cos(t \log b) = \frac{1}{2} \Re \left(\frac{a}{b} \right)^{i t} + \frac{1}{2}\Re (ab)^{i t}$ and substituting $y = (\gamma-t)\log x$, this is in turn
    \begin{equation}\begin{split}
        &= \frac{1}{2 \log x} \Re \bigg[ \bigg( \sum_{T < \gamma \le 2 T} \left( \frac{a}{b} \right)^{i \gamma} \bigg) \int_{\mathbb{R}} h(y) e^{-i y \log a/b / \log x} \, dy \bigg] \\ &+ \frac{1}{2 \log x} \Re \bigg[ \bigg( \sum_{T < \gamma \le 2 T} (ab)^{i \gamma} \bigg) \int_{\mathbb{R}} h(y) e^{-i y \log ab / \log x} \, dy \bigg] + O \left( \sqrt{T} \log T \right) \\
        &= \frac{\hat h \left( \frac{\log a/b}{2 \pi \log x} \right)}{2 \log x} \Re \sum_{T < \gamma \le 2 T} \left( \frac{a}{b} \right)^{i \gamma} + \frac{\hat h \left( \frac{\log a b}{2 \pi \log x} \right)}{2 \log x} \Re \sum_{T < \gamma \le 2 T} (a b)^{i \gamma} + O \left( \sqrt{T} \log T \right).
    \end{split}\end{equation}
    Now, we continue by isolating the contribution of the diagonal $a = b$ of the first summand here to $\mathcal{I}_{P^2 \gamma}$. Since
    \[ \frac{1}{T} \sum_{T < \gamma \le 2 T} 1 = \frac{\log T+C}{2 \pi} + O \left( \frac{\log T}{T} \right), \]
    this contribution matches exactly the main term in \eqref{Psqhath}. Hence, it now suffices to show that the remaining contributions to $\mathcal{I}_{P^2 \gamma}$ are small, that is to say we need
    \begin{equation}\label{6may.3}
        \frac{1}{T} \sum_{a \ne b \le x} \frac{\Lambda(a) \Lambda(b)}{\sqrt{ab} (\log a) (\log b)} f \left( \frac{\log a}{\log x} \right) f \left( \frac{\log b}{\log x} \right) \hat h \left( \frac{\log a/b}{\log x} \right) \Re \sum_{T < \gamma \le 2 T} \left( \frac{a}{b} \right)^{i \gamma} \ll 1
    \end{equation}
    and
    \begin{equation}\label{6may.5}
        \frac{1}{T} \sum_{a, b \le x} \frac{\Lambda(a) \Lambda(b)}{\sqrt{ab} (\log a) (\log b)} f \left( \frac{\log a}{\log x} \right) f \left( \frac{\log b}{\log x} \right) \hat h \left( \frac{\log a b}{\log x} \right) \Re \sum_{T < \gamma \le 2 T} (ab)^{i \gamma} \ll 1
    \end{equation}
    as well as
    \begin{equation}
        \frac{\log T}{\sqrt{T}} \sum_{a, b \le x} \frac{\Lambda(a) \Lambda(b)}{\sqrt{ab} (\log a) (\log b)} f \left( \frac{\log a}{\log x} \right) f \left( \frac{\log b}{\log x} \right) \ll \frac{1}{\log x}.
    \end{equation}
    This last claim is obvious by using our usual bound, which in fact gives that the left-hand side is $\ll \frac{x (\log T)}{\sqrt{T}}$.

    As for \eqref{6may.3}, we will restrict to $b < a$, the other range follows in the same way. Then Lemma \ref{GoneksFormula} gives that
    \begin{equation}\label{6may.4}
        \frac{1}{T} \Re \sum_{T < \gamma \le 2 T} \left( \frac{a}{b} \right)^{i \gamma} = - \frac{\Lambda(a/b)}{2 \pi} \sqrt{\frac{b}{a}} + O \left( \frac{x (\log T)^2}{T} \right).
    \end{equation}
    Firstly, the error term in this bound gives a contribution to \eqref{6may.3} of at most
    \[ \ll \frac{x(\log T)^2}{T} \sum_{a,  b \le x} \frac{\Lambda(a) \Lambda(b)}{\sqrt{ab} (\log a) (\log b)} \ll \frac{x^2 (\log T)^2}{T}, \]
    which is acceptable. The main term of \eqref{6may.4} only contributes when $\frac{a}{b}$ is a prime power, but since each of them are themselves prime powers, we must have $a = p^\alpha, b = p^\beta$ for some $\alpha > \beta \ge 1$. This gives a contribution to \eqref{6may.3} of
    \[  \ll \sum_{\substack{p, \alpha, \beta \\ \alpha > \beta \ge 1 }} \frac{\log p}{\alpha \beta p^\alpha} \ll 1, \]
    so that \eqref{6may.3} follows. The proof of \eqref{6may.5} works in the same way, the only difference being that we apply Lemma \ref{GoneksFormula} with numerator $ab \le x^2$ and denominator $1$.
\end{proof}

Arguing similarly, we get the analogous result for the imaginary part.

\begin{proposition}
    Assume RH. Uniformly for $2 \le x \le T^{1/4}$, we have
    \begin{equation}\notag
        \frac{1}{T} \int_T^{2 T} \mathcal P(t)^2 \mathcal Z(t) \, dt = O \left( \frac{1}{\log x} \right).
    \end{equation}
\end{proposition}

\begin{proof}
    The proof follows closely that of Proposition \ref{P^2Z_Re}. With the same strategy we used above to handle $\mathcal I_{P^2\gamma}$, by applying Lemma \ref{6may.2} part \eqref{a}, we obtain 
    \begin{equation}\begin{split}\notag
        \frac{1}{T} \int_T^{2 T} \mathcal P(t)^2 \mathcal Z(t) \, dt
        = &\frac{1}{T}\sum_{a,b\leq x}\frac{\Lambda(a)\Lambda(b)}{\sqrt{ab}\log a\log b} \mathfrak f\bigg(\frac{\log a}{\log x}\bigg) \mathfrak f\bigg(\frac{\log b}{\log x}\bigg) \\
        &\times \int_{\mathbb R} \sum_{T\leq \gamma \leq 2T} \mathfrak h((t-\gamma)\log x) \sin(t\log a)\sin(t\log b) dt
        + O\bigg(\frac{x\log T}{\sqrt{T}}\bigg).
    \end{split}\end{equation}
    Using the identity $\sin(t\log a)\sin(t\log b) = \frac{1}{2}\Re(ab)^{it}-\frac{1}{2}\Re(\frac{a}{b})^{it}$, 
    we can then perform the same calculations as in the previous proof, and the claim follows.
    \end{proof}

\section{The contribution from one prime power and two zeros}

We introduce the following notation:
\begin{equation}\label{LDef}\mathcal L(b) := \frac{(b+1)\log(1+b)+(1-b)\log(1-b)}{b}.
\end{equation}

\begin{proposition}
    Assume RH and Conjecture \ref{TPCC}. Uniformly for $2 \le x \le T^{1/3}$, we have
    \begin{equation}\notag
        \frac{1}{T}\int_T^{2T} P(t)Z(t)^2 dt = \frac{1}{2} \int_0^\beta  \bigg(\frac{g(\frac{b}{\beta})}{\beta}-\frac{1}{b}\bigg) \mathcal L(b) \,db
        + O\bigg(\frac{1}{\log x}\bigg),
    \end{equation}
    where $g$ is defined in \eqref{DefAuxFunctions} and $\mathcal L$ in \eqref{LDef}.
\end{proposition}

\begin{proof}
    According to the definition of $Z(t)$ in the statement of Proposition \ref{P&Z}, we write
    \begin{equation}\label{1oct.10}
        \frac{1}{T}\int_T^{2T} P(t)Z(t)^2 dt 
        = \mathcal I_{P\hat h^2} - 2\mathcal I_{P\hat h\gamma} + \mathcal I_{P\gamma^2}
    \end{equation}
    with
    \begin{equation}\begin{split}\notag
        \mathcal I_{P\hat h^2}
        &= \frac{\hat h(0)^2}{(2\pi\beta)^2}
        \sum_{n\leq x} \frac{\Lambda(n)}{\sqrt{n}\log n} f\bigg(\frac{\log n}{\log x} \bigg) \frac{1}{T}\int_T^{2T} \cos(t\log n) \bigg ( \frac{\log\frac{t}{2\pi}}{\log T}\bigg)^2  dt \\
        \mathcal I_{P\hat h\gamma}
        &= \frac{\hat h(0)}{2\pi\beta}
        \sum_{n\leq x} \frac{\Lambda(n)}{\sqrt{n}\log n} f\bigg(\frac{\log n}{\log x} \bigg) \frac{1}{T}\int_T^{2T} \cos(t\log n) \frac{\log\frac{t}{2\pi}}{\log T} \sum_{\gamma} h((\gamma-t)\log x)  dt\\
        \mathcal I_{P\gamma^2}
        &= \sum_{n\leq x} \frac{\Lambda(n)}{\sqrt{n}\log n} f\bigg(\frac{\log n}{\log x} \bigg) \frac{1}{T}\int_T^{2T} \cos(t\log n) \sum_{\gamma,\gamma'} h((\gamma-t)\log x) h((\gamma'-t)\log x)  dt .
    \end{split}\end{equation}
    A direct application of Lemma \ref{CosLogInt} yields
    \begin{equation}\label{1oct.11}
        \mathcal I_{P\hat h^2}
        \ll \frac{1}{T} \sum_{n\leq x} \frac{\Lambda(n)}{\sqrt{n}(\log n)^2} 
        \ll \frac{\sqrt{x}(\log T)^2}{T}.
    \end{equation}
    We now turn to $\mathcal I_{P\hat h\gamma}$. 
    Let $\varphi$ be a function such that $\varphi(t)=\log(\frac{t}{2\pi})/\log T$ for $T\leq t\leq 2T$, $\varphi(t)=1+O(1/\log T)$ for any $t$, and $\varphi'(t)\ll 1/((1+|t|)\log T)$
    uniformly for $t\in\mathbb R$. For example, we can take a (fixed) smooth function $\tilde \varphi$ with 
    \[ \tilde \varphi(s) = \log \frac{s}{2 \pi} \text{ for } s \in [1,2] \text{ and } \tilde \varphi(s) = 0 \text{ for } s \not \in \left[ \frac{1}{2},\frac{5}{2} \right] \]
    and set
    \[ \varphi(t) = 1 + \frac{\tilde \varphi(t/T)}{\log T}. \]
    We then write
    \begin{equation}\begin{split}\notag
        \mathcal I_{P\hat h\gamma}
        &= \frac{\hat h(0)}{2\pi\beta}
        \sum_{n\leq x} \frac{\Lambda(n)}{\sqrt{n}\log n} f\bigg(\frac{\log n}{\log x} \bigg) \frac{1}{T}\int_T^{2T} \Re(n^{it}) \varphi(t) \sum_{\gamma} h((\gamma-t)\log x)  dt\\
        &= \Re\bigg(\frac{\hat h(0)}{2\pi\beta} \frac{1}{T\log x} 
        \sum_{n\leq x} \frac{\Lambda(n)}{\sqrt{n}\log n} f\bigg(\frac{\log n}{\log x} \bigg) \int_{\mathbb R} h(y) n^{-\frac{iy}{\log x}} \sum_{T\leq \gamma\leq 2T} n^{i\gamma} \varphi\bigg(\gamma-\frac{y}{\log x}\bigg) dy \bigg) \\
        & \hspace{12cm} + O\bigg(\frac{\sqrt{x}\log T}{\sqrt{T}}\bigg)
    \end{split}\end{equation}
    by Lemma \ref{6may.2} part \eqref{a}, and the change of variable $(\gamma-t)\log x=y$.
    Using Lemma \ref{GoneksFormula} and partial summation, one has
    \begin{equation}\begin{split}\notag
        \sum_{T\leq \gamma\leq 2T} n^{i\gamma} \varphi\bigg(\gamma-\frac{y}{\log x}\bigg)
        = -\frac{T}{2\pi}\frac{\Lambda(n)}{\sqrt n} + O(x(\log T)^2) + O\bigg(\frac{T\Lambda(n)}{\sqrt{n}\log T}\bigg)
    \end{split}\end{equation}
    uniformly for $y\in\mathbb R$.
    Hence, since $x\leq T^{1/3}$
    \begin{equation}\begin{split}\label{1oct.12}
        \mathcal I_{P\hat h\gamma}
        = -\frac{\hat h(0)}{(2\pi\beta)^2} \frac{1}{\log T} 
        \sum_{n\leq x} \frac{\Lambda(n)^2}{n\log n} f\bigg(\frac{\log n}{\log x} \bigg) \hat h\bigg(\frac{\log n}{2\pi\log x}\bigg) + O\bigg(\frac{1}{\log T}\bigg). 
    \end{split}\end{equation}
    Finally, we deal with $\mathcal I_{P\gamma^2}$. By Lemma \ref{6may.2} part \eqref{b} and the change of variable $(\gamma-t)\log x=y$, we obtain
    \begin{equation}\begin{split}\notag
        \mathcal I_{P\gamma^2}
        & = \Re\bigg( \frac{1}{T\log x}  \sum_{n\leq x} \frac{\Lambda(n)}{\sqrt{n}\log n} f\bigg(\frac{\log n}{\log x} \bigg) \sum_{T\leq \gamma,\gamma'\leq 2T} n^{i\gamma}\\
        & \hspace{2cm} \times\int_{\mathbb R}n^{-\frac{iy}{\log x}}h(y)h(y+(\gamma'-\gamma)\log x) dy\bigg) 
        + O\bigg( \frac{\sqrt{x}(\log T)^2}{\sqrt{T}}\bigg). 
    \end{split}\end{equation}
    The error term is clearly $\ll T^{-1/4}$, say. By introducing the Fourier pair
    \begin{equation}\notag
        \hat k_n(u) = \int_{\mathbb R} n^{-\frac{iy}{\log x}} h(y) h(y+u) dy, 
        \quad k_n(a) = \hat h(a) \hat h\bigg(-a+\frac{\log n}{2\pi\log x}\bigg),
    \end{equation}
    we write
    \begin{equation}\begin{split}\notag
        \mathcal I_{P\gamma^2}
        & = \Re\bigg( \frac{1}{T\log x}  \sum_{n\leq x} \frac{\Lambda(n)}{\sqrt{n}\log n} f\bigg(\frac{\log n}{\log x} \bigg) \sum_{T\leq \gamma,\gamma'\leq 2T} n^{i\gamma} \hat k_n((\gamma'-\gamma)\log x) \bigg) 
        + O(T^{-1/4}). 
    \end{split}\end{equation}
    Now we appeal to the Twisted Pair Correlation Conjecture to evaluate the double sum over zeros. Namely, we apply Conjecture \ref{TPCC} with $r(u)=\hat k_{n}(2\pi\beta u)$ and  $\hat r(a)=\frac{1}{2\pi\beta}k_{n}(\frac{-a}{2\pi\beta})$, and obtain
    \begin{equation}\begin{split}\label{1oct.13}
        \mathcal I_{P\gamma^2}
        & = -\frac{(\log T)^{-1}}{(2\pi\beta)^2}\sum_{n\leq x} \frac{\Lambda(n)^2}{n \log n} f\bigg(\frac{\log n}{\log x} \bigg) \bigg( 2k_n(0) + \int_{\mathbb R} k_n \bigg(\frac{-a}{2\pi\beta}\bigg) m_n(a) da \bigg)
        + O\bigg(\frac{1}{\log x}\bigg)
    \end{split}\end{equation}
    since $k_n$ is real over reals. We highlight the fact the contributions  
    from the first term of the error $\mathcal E_n$ can be clearly absorbed in the error term $O(1/\log x)$, as it becomes small on average over $n$. Putting together Equations \eqref{1oct.10}-\eqref{1oct.13}, we notice that the $k_n(0)$-term above cancels the main term in $\mathcal I_{P\hat h\gamma}$, and we have
    \begin{equation}\begin{split}\notag
        \frac{1}{T}\int_T^{2T} P(t)Z(t)^2 dt 
        &= -\frac{(\log T)^{-1}}{(2\pi\beta)^2}\sum_{n\leq x} \frac{\Lambda(n)^2}{n \log n} f\bigg(\frac{\log n}{\log x} \bigg) \int_{\mathbb R} k_n\bigg(\frac{-a}{2\pi\beta}\bigg) m_n(a) da
        + O\bigg(\frac{1}{\log x}\bigg).
    \end{split}\end{equation}
    Since $m_n(a)=0$ if $a\in(-1,1-\log n/\log T)$, in the integral above we can replace $k_n(\frac{-a}{2\pi\beta})$ by $\frac{(\pi\beta)^2}{|a||a+\log n/\log T|}$, in view of \eqref{PropertiesAuxiliaryFunctions}. Moreover, for any $1<n<T^{1/2}$, by direct computation one can see that 
    \begin{equation}\notag
        \int_{\mathbb R}  \frac{m_n(a)}{|a| |a+\frac{\log n}{\log T}|}  da
        = 2\frac{(\frac{\log n}{\log T}+1)\log(1+\frac{\log n}{\log T})+(1-\frac{\log n}{\log T}) \log(1-\frac{\log n}{\log T})}{(\frac{\log n}{\log T})^2} 
        = \frac{2 \mathcal L(\frac{\log n}{\log T})}{\frac{\log n}{\log T}}.
    \end{equation}
    Hence,
    \begin{equation}\begin{split}\notag
        \frac{1}{T}\int_T^{2T} P(t)Z(t)^2 dt 
        &= -\frac{1}{2}\sum_{n\leq x} \frac{\Lambda(n)^2}{n (\log n)^2} f\bigg(\frac{\log n}{\log x} \bigg) \mathcal L\bigg(\frac{\log n}{\log T}\bigg)
        + O\bigg(\frac{1}{\log x}\bigg).
    \end{split}\end{equation}
    Integrating by parts twice, and using in between the well-known formula (conditional on RH)
    $$\sum_{n\leq v} \frac{\Lambda(n)^2}{n(\log n)^2} = \log\log v + C + O(v^{-1/2+\varepsilon})$$
    for an explicit constant $C$ and any small $\varepsilon>0$, we obtain 
    \begin{equation}\notag
        \frac{1}{T}\int_T^{2T} P(t)Z(t)^2 dt 
        = -\frac{1}{2}\int_{2}^{x} f\bigg(\frac{\log v}{\log x} \bigg) \mathcal L\bigg(\frac{\log v}{\log T}\bigg) \frac{dv}{v\log v}
        + O\bigg(\frac{1}{\log x}\bigg).
    \end{equation}
    The claim follows by change of variable $\log v=b\log T$, and using $f(x)=1-xg(x)$.
\end{proof}

The analogous result for the imaginary part reads:

\begin{proposition}
    Assume RH and Conjecture \ref{TPCC}. Uniformly for $2 \le x \le T^{1/3}$, we have
    \begin{equation}\notag
        \frac{1}{T}\int_T^{2T} \mathcal P(t) \mathcal Z(t)^2 dt \ll \frac{1}{\log x}.
    \end{equation}
\end{proposition}

\begin{proof}
    With a now familiar strategy, we write
    \begin{equation}\notag
        \frac{1}{T}\int_T^{2T} \mathcal P(t) \mathcal Z(t)^2 dt 
        = \Im\bigg(\frac{-1}{T\log x} \sum_{n\leq x} \frac{\Lambda(n)}{\sqrt{n}\log n} \mathfrak f\bigg(\frac{\log n}{\log x}\bigg) \sum_{T\leq \gamma,\gamma'\leq 2T} n^{i\gamma} \hat{\kappa}_n((\gamma-\gamma')\log x) \bigg) + O(T^{-1/4}),
    \end{equation}
    with
    \begin{equation}\notag
        \hat\kappa_n(u)
        = \int_{\mathbb R} n^{\frac{iy}{\log x}} \mathfrak h(y)
        \mathfrak h(y+u) dy,
        \quad \kappa_n(a) = \hat{\mathfrak h}(-a) \hat{\mathfrak h}\bigg(a-\frac{\log n}{2\pi\log x}\bigg) .
    \end{equation}
    The claim then follows from an application of Conjecture \ref{TPCC}. Indeed the resulting main term is real (and then disappears when we take the imaginary part), and the error term contributes $\ll 1/\log x$.
\end{proof}

\section{The contribution from three zeros}\label{SecFinalMainEstimates}

\begin{proposition}
    Assume RH, Conjecture \ref{PCC}, and Conjecture \ref{TCC2}. Then, uniformly for $2\leq x\leq T$, we have
    \begin{equation}\notag
        \frac{1}{T} \int_T^{2 T} Z(t)^3 \, dt 
        = - \frac{\pi^2}{4} 
        - \frac{3}{2} \int_{0}^\beta \bigg( \frac{g(\frac{b}{\beta})}{\beta} -\frac{1}{b}\bigg) \mathcal L(b) db + O\bigg(\frac{1}{\log x}\bigg)
    \end{equation}
    where $g$ is defined in \eqref{DefAuxFunctions} and $\mathcal L$ in \eqref{LDef}.
\end{proposition}

\begin{proof}
With our now standard procedure, in view of Proposition \ref{P&Z} we write
    \begin{equation}\label{21may.00}
        \frac{1}{T} \int_T^{2 T} Z(t)^3 dt 
        = \mathcal I_{\hat h^3} - 3\mathcal I_{\hat h^2\gamma} + 3\mathcal I_{\hat h\gamma^2} - \mathcal I_{\gamma^3}
    \end{equation}
    with
    \begin{align}\label{21may.01}
        \mathcal I_{\hat h^3}
        &=  \frac{\hat h(0)^3}{(2\pi\beta)^3}\frac{1}{T} \int_T^{2T} \bigg(\frac{\log\frac{t}{2\pi}}{\log T}\bigg)^3 dt
        = \frac{\hat h(0)^3}{(2\pi\beta)^3}+ O\bigg(\frac{1}{\log T}\bigg)\\ \notag
        \mathcal I_{\hat h^2\gamma}
        &=  \frac{\hat h(0)^2}{(2\pi\beta)^2}\frac{1}{T} \int_T^{2T} \bigg(\frac{\log\frac{t}{2\pi}}{\log T}\bigg)^2 \sum_\gamma h((\gamma-t)\log x) dt\\ \notag
        \mathcal I_{\hat h\gamma^2}
        &=  \frac{\hat h(0)}{2\pi\beta}\frac{1}{T} \int_T^{2T} \frac{\log\frac{t}{2\pi}}{\log T} \bigg(\sum_\gamma h((\gamma-t)\log x)\bigg)^2 dt\\ \notag
        \mathcal I_{\gamma^3}
        &= \frac{1}{T} \int_T^{2T} \bigg(\sum_\gamma h((\gamma-t)\log x)\bigg)^3 dt.
    \end{align}
We start by $\mathcal I_{\hat h\gamma}$, for which we use the trivial estimate $\log\frac{t}{2\pi}=\log T+O(1)$, for $T\leq t\leq 2T$. Applying Lemma \ref{6may.2} part \eqref{a} with $\phi=h$ and $\phi=|h|$, we have
    \begin{equation}\notag
        \mathcal I_{\hat h^2\gamma}
        = \frac{\hat h(0)^2}{(2\pi\beta)^2}\frac{1}{T} \int_{\mathbb R} \bigg(1+O\bigg(\frac{1}{\log T}\bigg)\bigg) \sum_{T\leq\gamma\leq 2T} h((\gamma-t)\log x) dt + O\bigg(\frac{\log T}{\sqrt{T}}\bigg) .
    \end{equation}
    The first error term above contributes
    \begin{equation}\notag
        \ll \frac{1}{T\log T} \int_{\mathbb R} \sum_{T\leq\gamma\leq 2T} |h((\gamma-t)\log x)| dt
        \ll \frac{1}{T\log T} \int_{\mathbb R} \sum_{T\leq\gamma\leq 2T} |h(y)| \frac{dy}{\log x}
        \ll \frac{1}{\log x}.
    \end{equation}
    since $h\in L^1(\mathbb R)$. Therefore, by the change of variable $(\gamma-t)\log x=y$ and an application of the Riemann–von Mangoldt formula, we get
    \begin{equation}\label{21may.02}
        \mathcal I_{\hat h^2\gamma}
        = \frac{\hat h(0)^2}{(2\pi\beta)^2}\frac{\log T}{2\pi\log x} \int_{\mathbb R} h(y) dy
        + O\bigg(\frac{1}{\log x}\bigg) + O\bigg(\frac{\log T}{\sqrt{T}}\bigg)\\
        = \frac{\hat h(0)^3}{(2\pi\beta)^3}
        + O\bigg(\frac{1}{\log x}\bigg).
    \end{equation}
    Arguing similarly, by Lemma \ref{6may.2} part \eqref{b}, one gets
    \begin{equation}\begin{split}\notag
       \mathcal I_{\hat h\gamma^2}
       = &\frac{\hat h(0)}{2\pi\beta}\frac{1}{T} \int_{\mathbb R} h(y)  \sum_{T\leq\gamma,\gamma'\leq 2T} h(y+(\gamma'-\gamma)\log x) \frac{dy}{\log x} \\
       &+ O\bigg(\frac{1}{T\log x} \int_{\mathbb R} |h(y)| \sum_{T\leq\gamma,\gamma'\leq 2T} |h(y+(\gamma'-\gamma)\log x)| \frac{dy}{\log x} \bigg)
       +O\bigg(\frac{(\log T)^2}{\sqrt{T}}\bigg).
    \end{split}\end{equation}
    To write the main term of $\mathcal I_{\hat h\gamma^2}$ more concisely, we denote
    $$\hat k(u) = \int_{\mathbb R} h(y)h(y+u)dy, 
    \quad k(a) = \hat h(a)\hat h(-a) = \hat h(a)^2 $$
    since $h$ is even. Similarly, to deal with the first error term we introduce 
    $$\widehat{k_*}(u) = \int_{\mathbb R} |h(y)||h(y+u)|dy, 
    \quad k_*(a) = \widehat{|h|}(a)\widehat{|h|}(-a) = \widehat{|h|}(a)^2 .$$
    In these notations, we have
    \begin{equation}\begin{split}\notag
       \mathcal I_{\hat h\gamma^2}
       = &\frac{\hat h(0)}{2\pi\beta}\frac{1}{T\log x} \sum_{T\leq\gamma,\gamma'\leq 2T} \hat k((\gamma'-\gamma)\log x)\\
       &+ O\bigg(\frac{1}{T(\log x)^2}  \sum_{T\leq\gamma,\gamma'\leq 2T} \widehat{k_*}((\gamma'-\gamma)\log x) \bigg)
       +O\bigg(\frac{(\log T)^2}{\sqrt{T}}\bigg).
    \end{split}\end{equation}
    By applying of Conjecture \ref{PCC} with $r(u)=\hat k(2\pi\beta u)$, $\hat r(a)=\frac{1}{2\pi\beta}k(\frac{-a}{2\pi\beta})$ and with $r(u)=\widehat{k_*}(2\pi\beta u)$, $\hat r(a)=\frac{1}{2\pi\beta}k_*(\frac{-a}{2\pi\beta})$, we obtain
    \begin{equation}\begin{split}\label{21may.03}
       \mathcal I_{\hat h\gamma^2}
       &= \frac{\hat h(0)}{(2\pi\beta)^3}
       \int_{\mathbb R} k\bigg(\frac{-a}{2\pi\beta}\bigg) \bigg(\delta(a) + \min\{|a|,1\}\bigg) da
       +O\bigg(\frac{1}{\log x}\bigg)\\
       &= \frac{\hat h(0)^3}{(2\pi\beta)^3}
       +\frac{\hat h(0)}{(2\pi\beta)^3}
       \int_{\mathbb R} \hat h\bigg(\frac{a}{2\pi\beta}\bigg)^2 \min\{|a|,1\} da
       +O\bigg(\frac{1}{\log x}\bigg).
    \end{split}\end{equation} 
    Finally, by assuming Conjecture \ref{TCC2} we deal with $\mathcal I_{\gamma^3}$. Denote
    $$\hat k(u,v) = \int_{\mathbb R} h(y)h(y+u)h(y+v)dy, \quad 
    k(a,b) = \hat h(-a) \hat h(-b) \hat h(a+b). $$
    By Lemma \ref{6may.2} part \eqref{c} and the usual manipulations, we write
    \begin{equation}\notag
        \mathcal I_{\gamma^3} 
        = \frac{1}{T\log x}\sum_{T\leq \gamma,\gamma',\gamma''\leq 2T} \hat k((\gamma'-\gamma)\log x, (\gamma''-\gamma)\log x) 
        + O\bigg(\frac{(\log T)^3}{\sqrt{T}}\bigg)
    \end{equation}
    Now we apply Conjecture \ref{TCC2} with $r(u,v) = \hat k(2\pi\beta u,2\pi\beta v)$, $\hat r(a,b) =\frac{1}{(2\pi\beta)^2}k(\frac{-a}{2\pi\beta},\frac{-b}{2\pi\beta})$, we have
    \begin{equation}\begin{split}\notag
        \mathcal I_{\gamma^3} 
        &= \frac{1}{(2\pi\beta)^3} \int_{\mathbb R}\int_{\mathbb R} k\bigg(\frac{-a}{2\pi\beta},\frac{-b}{2\pi\beta}\bigg) H(a,b) \, da \,db
        + O\bigg(\frac{1}{\log x}\bigg).
    \end{split}\end{equation}
    According to the notations used in \eqref{6may.20}, we write
    \begin{equation}\label{21may.04}
        \mathcal I_{\gamma^3} = \mathcal I_{\gamma^3,\delta} + \mathcal I_{\gamma^3,*} + O\bigg(\frac{1}{\log x}\bigg)
    \end{equation}
    with
    \begin{equation}\notag
        \mathcal I_{\gamma^3,\delta} 
        = \frac{1}{(2\pi\beta)^3} \int_{\mathbb R}\int_{\mathbb R} \hat h\bigg(\frac{a}{2\pi\beta}\bigg) \hat h\bigg(\frac{b}{2\pi\beta}\bigg) \hat h\bigg(\frac{a+b}{2\pi\beta}\bigg) H_{\delta}(a,b) \, da \,db 
    \end{equation}
    and
    \begin{equation}\begin{split}\notag
        \mathcal I_{\gamma^3,*} 
        = \frac{1}{(2\pi\beta)^3} \int_{\mathbb R}\int_{\mathbb R} \hat h\bigg(\frac{a}{2\pi\beta}\bigg) \hat h\bigg(\frac{b}{2\pi\beta}\bigg) \hat h\bigg(\frac{a+b}{2\pi\beta}\bigg) H_*(a,b) \, da \,db.
    \end{split}\end{equation}
    Evaluating the integrals involving Dirac delta functions and using the parity of $h$, one easily sees that
    \begin{equation}\label{21may.05}
        \mathcal I_{\gamma^3,\delta} 
        = \frac{\hat h(0)^3}{(2\pi\beta)^3} + \frac{3\hat h(0)}{(2\pi\beta)^3} \int_{\mathbb R} \hat h \bigg(\frac{a}{2\pi\beta}\bigg)^2 \min\{|a|,1\}da.
    \end{equation}
    Plugging Equations \eqref{21may.01}-\eqref{21may.05} into \eqref{21may.00}, we notice that second term in above expression cancels the second term on the second line of \eqref{21may.03}. Similarly, all the terms $\frac{\hat h(0)^3}{(2\pi\beta)^3}$ cancel. Namely, we obtain
    \begin{equation}\label{21may.06}
        \frac{1}{T} \int_T^{2 T} Z(t)^3 dt 
        = - \mathcal I_{\gamma^3,*} +  O\bigg(\frac{1}{\log x}\bigg) .
    \end{equation}
    Evaluating $\mathcal I_{\gamma^3,*}$ is the last step of the proof. For starters, we note that by the various symmetries of the functions involved, one sees that 
    \begin{equation}\label{4sep.1}
        \mathcal I_{\gamma^3,*} 
        = \frac{6}{(2\pi\beta)^3} \int_0^\infty \int_0^\infty \hat h\bigg(\frac{a}{2\pi\beta}\bigg) \hat h\bigg(\frac{b}{2\pi\beta}\bigg)  \hat h\bigg(\frac{a+b}{2\pi\beta}\bigg) H_*(a,b) db\,da. 
    \end{equation}
    To prove \eqref{4sep.1}, it suffices to split the plane into the six regions delimited by the axes and the bisector of second and fourth quadrant. Since the integrand function in the definition of $\mathcal I_{\gamma^3,*}$ is invariant under the transformations $(a,b)\to(-a,-b)$, $(a,b)\to(-a-b,b)$, and $(a,b)\to(-a-b,b)$, the integrals over each region equals (say) the integral over the first quadrant, and \eqref{4sep.1} is then verified.
    Now we note that in the first quadrant $H_*(a,b)=0$ if $a+b<1$; in the complementary sub-region $a+b>1$ we have $G(a,b)=0$ and then $H_*(a,b)=\min\{a,1\} + \min\{b, 1\} -1$. With \eqref{PropertiesAuxiliaryFunctions} in mind, we integrate the function $\frac{\pi\beta}{\cdot}$ instead of $\hat h(\frac{\cdot}{2\pi\beta})$ and then we correct this discrepancy. The correction terms involve integrals over $R_a=\{(a,b): 0<b<\beta \text{ and } a>1-b \}$ and $R_b=\{(a,b): 0<a<\beta \text{ and } b>1-a\}$. Namely:
    \begin{equation}\begin{split}\notag
        \mathcal I_{\gamma^3,*} 
        & = \frac{6}{(2\pi\beta)^3} \int_0^\infty \int_0^\infty \frac{\pi\beta}{a} \frac{\pi\beta}{b} \frac{\pi\beta}{a+b}  H_*(a,b) db\,da\\
        &\quad+ \frac{6}{(2\pi\beta)^3} \iint_{R_a} \frac{\pi\beta}{a} \bigg(\pi g\bigg(\frac{b}{\beta}\bigg) -\frac{\pi\beta}{b}\bigg) \frac{\pi\beta}{a+b} H_*(a,b) db\,da\\
        &\quad + \frac{6}{(2\pi\beta)^3} \iint_{R_b}  \bigg(\pi g\bigg(\frac{a}{\beta}\bigg) -\frac{\pi\beta}{a}\bigg) \frac{\pi\beta}{b} \frac{\pi\beta}{a+b} H_*(a,b) db\,da.
    \end{split}\end{equation}
    By symmetry, the integrals over $R_a$ and $R_b$ are the same, so
    \begin{equation}\begin{split}\label{21may.07}
        \mathcal I_{\gamma^3,*} 
        & = \frac{3}{4} \int_0^\infty \int_0^\infty \frac{H_*(a,b)}{ab(a+b)} db\,da
        + \frac{3}{2} \iint_{R_a} \bigg( \frac{g(\frac{b}{\beta})}{\beta} -\frac{1}{b}\bigg) \frac{H_*(a,b)}{a(a+b)}  db\,da.
    \end{split}\end{equation}
    The first integral is straightforward:
    \begin{equation}\begin{split}\notag
        \int_0^\infty \int_0^\infty \frac{H_*(a,b)}{ab(a+b)} db\,da = \frac{\pi^2}{3}.
    \end{split}\end{equation}
    For the second integral in \eqref{21may.07}, we notice that on $R_a$ we have $H_*(a,b) = b + \min\{a,1\} -1$, and we get
    \begin{equation}\begin{split}\notag
        \mathcal I_{\gamma^3,*} 
        & = \frac{\pi^2}{4} 
        + \frac{3}{2} \int_{0}^\beta \bigg( \frac{g(\frac{b}{\beta})}{\beta} -\frac{1}{b}\bigg) \int_{1-b}^{\infty} \frac{b+\min\{a,1\}-1}{a(a+b)}  da\,db\\
        & = \frac{\pi^2}{4} 
        + \frac{3}{2} \int_{0}^\beta \bigg( \frac{g(\frac{b}{\beta})}{\beta} -\frac{1}{b}\bigg) \mathcal L(b) db.
    \end{split}\end{equation}
    Plugging the above into \eqref{21may.06}, we conclude the proof.
\end{proof}

\begin{proposition}
    Assume RH and Conjecture \ref{TCC2}. Then, uniformly for $2\leq x\leq T$, we have
    \begin{equation}\notag
        \frac{1}{T} \int_T^{2 T} \mathcal Z(t)^3 \, dt  \ll \frac{1}{\log x}.
    \end{equation}
\end{proposition}

\begin{proof}
    We write
    \begin{equation}\notag
        \frac{1}{T} \int_T^{2 T} \mathcal Z(t)^3 \, dt  
        =  \frac{1}{T\log x} \sum_{T\leq \gamma,\gamma',\gamma''\leq 2T} \hat\kappa((\gamma-\gamma')\log x,(\gamma-\gamma'')\log x)
        + O(T^{-1/4}),
    \end{equation}
    with
    $$\hat\kappa(u,v) = \int_{\mathbb R}
    \mathfrak h(y) 
    \mathfrak h(y+u) 
    \mathfrak h(y+v) dy,
    \quad \kappa(a,b) = \hat{\mathfrak h}(-a) \hat{\mathfrak h}(-b) \hat{\mathfrak h}(a+b) .$$
    We apply Conjecture \ref{TCC2} and get
    \begin{equation}\notag
        \frac{1}{T} \int_T^{2 T} \mathcal Z(t)^3 \, dt  
        =  \frac{1}{(2\pi\beta)^3} \int_{\mathbb R} \int_{\mathbb R} \kappa\bigg(\frac{a}{2\pi\beta},\frac{b}{2\pi\beta}\bigg) H(a,b) da\,db
        + O\bigg(\frac{1}{\log x}\bigg).
    \end{equation}
    The main term above is 0 because of the symmetry of the integrand. Indeed, 
    $$H(a,b) = H(-a-b)$$
    and (since $\mathfrak h$ is odd)
    $$ \kappa\bigg(\frac{a}{2\pi\beta},\frac{b}{2\pi\beta}\bigg)
    = - \kappa\bigg(\frac{-a}{2\pi\beta},\frac{- b}{2\pi\beta}\bigg).$$
\end{proof}

\section{Proof of Proposition \ref{STPCCimpliesTPCC}}\label{SecSTPCCimpTCC}

To prove the result, we convolve the twisted pair correlation function $F_n$ with an appropriate kernel. More precisely, let $g\in L^1(\mathbb R)$ be such that $\hat g\in L^1(\mathbb R)$ is Lipschitz continuous and $\hat g(\alpha)\ll |\alpha|^{-3}$ as $|\alpha|\to\infty$. We multiply $F_n(\alpha)$ times $\hat g(\alpha)$, and then integrate over $\alpha$, obtaining
\begin{equation}\label{PROP.1}
    \int_{\mathbb R} \hat g(\alpha) F_n(\alpha) d\alpha 
    = -\bigg(\frac{T}{2\pi}\frac{\Lambda(n)}{\sqrt{n}} \bigg)^{-1}
    \sum_{T\leq\gamma,\gamma'\leq 2T} n^{i\gamma} \omega(\gamma-\gamma') g\bigg(\frac{\log T}{2\pi} (\gamma-\gamma') \bigg).
\end{equation}
First, we analyze the left hand side of \eqref{PROP.1}. We split the range of integration, and write 
\begin{equation}\label{PROP.2}
    \int_{\mathbb R} \hat g(\alpha) F_n(\alpha) d\alpha 
    = \bigg[ \int_{-\infty}^{-1} + \int_{-1}^{-\frac{\log n}{\log T}} + \int_{-\frac{\log n}{\log T}}^{0} + \int_{0}^{1-\frac{\log n}{\log T}} + \int_{1-\frac{\log n}{\log T}}^{+\infty} \bigg] \hat g(\alpha) F_n(\alpha) d\alpha.
\end{equation}
The five integrals above can be computed by applying Conjecture \ref{STPCC}. 
Since $\hat g \in L^1(\mathbb R)$, we immediately have
\begin{equation}\label{PROP.3}
    \int_{1-\frac{\log n}{\log T}}^{+\infty} \hat g(\alpha) F_n(\alpha) d\alpha 
    = \int_{1-\frac{\log n}{\log T}}^{+\infty} \hat g(\alpha) \min\bigg\{ 1, \frac{\log T}{\Lambda(n)}\bigg( \alpha - 1 + \frac{\log n}{\log T} \bigg) \bigg\} d\alpha 
    + O\bigg(\frac{1}{\log T}\bigg)
\end{equation}
and, recalling $F_n(\alpha)= F_n(-\alpha-\frac{\log n}{\log T})$,
\begin{equation}\begin{split}\label{PROP.7}
    \int_{-\infty}^{-1} \hat g&(\alpha) F_n(\alpha) d\alpha 
    =\int_{1-\frac{\log n}{\log T}}^{+\infty} \hat g\bigg(-\alpha-\frac{\log n}{\log T}\bigg) F_n(\alpha) d\alpha \\
    &= \int_{1-\frac{\log n}{\log T}}^{+\infty} \hat g\bigg(-\alpha-\frac{\log n}{\log T}\bigg) \min\bigg\{ 1, \frac{\log T}{\Lambda(n)}\bigg( \alpha - 1 + \frac{\log n}{\log T} \bigg) \bigg\} d\alpha 
    + O\bigg(\frac{1}{\log T}\bigg).
\end{split}\end{equation}
As for the range $(0,1-\frac{\log n}{\log T})$, an application of Conjecture \ref{STPCC} gives
\begin{equation}\begin{split}\label{PROP.4}
    \int_{0}^{1-\frac{\log n}{\log T}} \hat g(\alpha) F_n(\alpha) d\alpha 
    = \bigg(1+\frac{1}{n^2}&+O\bigg(\frac{1}{\log T}\bigg)\bigg)  \int_{0}^{1-\frac{\log n}{\log T}} \hat g(\alpha) T^{-2\alpha}\log T d\alpha \\
    & - \int_{0}^{1-\frac{\log n}{\log T}} \hat g(\alpha) r_1(\alpha,n) d\alpha 
    +O\bigg(\frac{1}{\log T}\bigg).
\end{split}\end{equation}
The first integral can be evaluated by Taylor expanding $\hat g$ around zero. Doing so, one gets
\begin{equation}\begin{split}\notag
    \int_{0}^{1-\frac{\log n}{\log T}} \hat g(\alpha) T^{-2\alpha}\log T d\alpha
    &= \hat g(0) \int_{0}^{1-\frac{\log n}{\log T}}  T^{-2\alpha}\log T d\alpha
    +O\bigg(\int_{0}^{1-\frac{\log n}{\log T}} \alpha T^{-2\alpha}\log T d\alpha \bigg) \\ 
    & = \frac{\hat g(0)}{2} + O\bigg(\frac{n^2}{T^2}\bigg)
    + O\bigg(\frac{1}{\log T}\bigg)
    = \frac{\hat g(0)}{2} + O\bigg(\frac{1}{\log T}\bigg).
\end{split}\end{equation}
Moreover, the second integral in \eqref{PROP.4} is small. Precisely, for $n=q^a$
\begin{equation}\begin{split}\notag
r_1(\alpha,n)
&= \frac{T^{-2\alpha}}{\Lambda(n)}\sum_{m\leq T^\alpha}m\Lambda(mn)\Lambda(m) 
+ \frac{T^{2\alpha}}{\Lambda(n)} \sum_{m>T^\alpha} \frac{\Lambda(mn)\Lambda(m)}{m^3} \\
&= T^{-2\alpha}\log q \sum_{b\leq \alpha\frac{\log T}{\log q}}q^b 
+ T^{2\alpha}\log q \sum_{b>\alpha\frac{\log T}{\log q}} q^{-3b},
\end{split}\end{equation}
and then
\begin{equation}\begin{split}\notag
    \int_{0}^{1-\frac{\log n}{\log T}} \hat g(\alpha)  r_1(\alpha,n) d\alpha 
    &\ll \log q \sum_{b\leq \frac{\log \frac{T}{n}}{\log q}} q^b \int_{\frac{b\log q}{\log T}}^{1-\frac{\log n}{\log T}} T^{-2\alpha} d\alpha
    + \log q \sum_{b=1}^{\infty} q^{-3b} \int_0^{\frac{b\log q}{\log T}} T^{2\alpha} d\alpha \\
    &\ll \frac{\log q}{\log T}  \sum_{b=1}^{\infty} q^{-b} 
    \ll \frac{\log q}{q\log T}  
    \ll \frac{1}{\log T}.
\end{split}\end{equation}
Therefore, \eqref{PROP.4} reads
\begin{equation}\begin{split}\label{PROP.5}
    \int_{0}^{1-\frac{\log n}{\log T}} \hat g(\alpha) F_n(\alpha) d\alpha 
    &= \frac{\hat g(0)}{2}\bigg(1+\frac{1}{n^2}\bigg)
    + O\bigg(\frac{1}{\log T}\bigg).
\end{split}\end{equation}
Arguing similarly, one has
\begin{equation}\begin{split}\label{PROP.6}
    \int_{-1}^{-\frac{\log n}{\log T}} \hat g(\alpha) F_n(\alpha) d\alpha 
    &= \int_0^{1-\frac{\log n}{\log T}} \hat g\bigg(-\alpha-\frac{\log n}{\log T}\bigg) F_n(\alpha) d\alpha \\
    &= \frac{\hat g(-\frac{\log n}{\log T})}{2}\bigg(1+\frac{1}{n^2}\bigg)
    + O\bigg(\frac{1}{\log T}\bigg).
\end{split}\end{equation}
Finally, we deal with the integral between $-\frac{\log n}{\log T}$ and $0$. Conjecture \ref{STPCC} leads to
\begin{equation}\begin{split}\notag
    \int_{-\frac{\log n}{\log T}}^{0} \hat g(\alpha) F_n(\alpha) d\alpha 
    &= \bigg(1+O\bigg(\frac{1}{\log T}\bigg)\bigg) \int_{-\frac{\log n}{\log T}}^0 \hat g(\alpha) \bigg(T^{2\alpha} + \frac{T^{-2\alpha}}{n^2} \bigg)\log T d\alpha \\
    & \quad -\int_{-\frac{\log n}{\log T}}^{0} \hat g(\alpha) r_2(\alpha,n) d\alpha
    + O\bigg(\frac{1}{\log T}\bigg).
\end{split}\end{equation}
The first term above can be computed by Taylor expanding $\hat g$. Namely, for the $T^{2\alpha}$-term in the first integral, one expands around 0. Similarly, for the $T^{-2\alpha}$-term, it suffices to expand around $-\frac{\log n}{\log T}$. By doing so one gets
\begin{equation}\begin{split}\notag
    \int_{-\frac{\log n}{\log T}}^{0} \hat g(\alpha) F_n(\alpha) d\alpha 
    = &\frac{\hat g(0)+\hat g(-\frac{\log n}{\log T})}{2} \bigg( 1-\frac{1}{n^2} \bigg)
    -\int_{-\frac{\log n}{\log T}}^{0} \hat g(\alpha) r_2(\alpha,n) d\alpha
    + O\bigg(\frac{1}{\log T}\bigg) .
\end{split}\end{equation}
To bound the remaining integral term above, we use a similar strategy as for $r_1(\alpha,n)$. Since $n=q^a$, the factors $\Lambda(m)$ and $\Lambda(n/m)$ force $m$ to be of the form $q^b$ for some $1\leq b\leq a-1$. As a consequence, 
\begin{equation}\begin{split}\notag
r_2(\alpha,n)
&= \frac{\log q}{(nT^\alpha)^2}\sum_{\substack{ 1\leq b <a \\ q^b\leq q^aT^\alpha }} q^{2b} 
+ (nT^{\alpha})^2 \log q \sum_{\substack{ 1\leq b <a \\ q^b> q^aT^\alpha }} q^{-2b}.
\end{split}\end{equation}
Hence, we have
\begin{equation}\begin{split}\notag
&\int_{-\frac{\log n}{\log T}}^{0} \hat g(\alpha) r_2(\alpha,n) d\alpha
= \int_{-\frac{\log n}{\log T}}^{0}\hat g(\alpha) \bigg(\frac{\log q}{(nT^\alpha)^2}\sum_{\substack{ 1\leq b <a \\ q^b\leq q^aT^\alpha }} q^{2b} 
+ (nT^{\alpha})^2 \log q \sum_{\substack{ 1\leq b <a \\ q^b> q^aT^\alpha }} q^{-2b}\bigg) d\alpha\\
&\ll \frac{\log q}{n^2} \sum_{1\leq b<a} q^{2b} \int_{-\frac{(a-b)\log q}{\log T}}^{0} T^{-2\alpha}d\alpha
+ n^2 \log q \sum_{1\leq b<a} q^{-2b} \int_{-\frac{a\log q}{\log T}}^{-\frac{(a-b)\log q}{\log T}} T^{2\alpha}d\alpha\\
&\ll \frac{\log q}{n^2} \sum_{1\leq b<a} q^{2b} \frac{n^2}{q^{2b}\log T}
+ n^2 \log q \sum_{1\leq b<a} q^{-2b} \frac{q^{2b}}{n^2\log T}
\ll \frac{\log q}{\log T} \sum_{1\leq b<a} 1
\ll (a-1)\frac{\log q}{\log T}
\end{split}\end{equation}
when $n=q^a$ with $a\geq 2$. Clearly, $r_2(\alpha,n)=0$ if $n$ is a prime, i.e. if $a=1$.
Therefore,
\begin{equation}\begin{split}\label{PROP.8}
    \int_{-\frac{\log n}{\log T}}^{0} \hat g(\alpha) F_n(\alpha) d\alpha 
    = \frac{\hat g(0)+\hat g(-\frac{\log n}{\log T})}{2} \bigg( 1-\frac{1}{n^2}\bigg) 
    + O\bigg(\frac{1}{\log T}\bigg) 
    + O (\mathcal E_n).
\end{split}\end{equation}
Plugging \eqref{PROP.3}-\eqref{PROP.8} into \eqref{PROP.2}, we obtain
\begin{equation}\begin{split}\notag
    \int_{\mathbb R} \hat g(\alpha) F_n(\alpha) d\alpha 
    =  &\hat g(0)+\hat g(-\tfrac{\log n}{\log T})
    + O (\mathcal E_n) \\
    & + \int_{1-\frac{\log n}{\log T}}^{+\infty} \bigg(\hat g(\alpha) + \hat g(-\alpha-\tfrac{\log n}{\log T}) \bigg) \min\bigg\{ 1, \frac{\log T}{\Lambda(n)}\bigg( \alpha - 1 + \frac{\log n}{\log T} \bigg) \bigg\} d\alpha ,
\end{split}\end{equation}
which in turn implies
\begin{equation}\begin{split}\label{PROP.10}
    \int_{\mathbb R} \hat g(\alpha) F_n(\alpha) d\alpha 
    =  \hat g(0)+\hat g(-\tfrac{\log n}{\log T})  
    + \int_{\mathbb R} &\frac{\hat g(\alpha) + \hat g(-\alpha-\tfrac{\log n}{\log T})}{2} m_n(\alpha)d\alpha 
    + O (\mathcal E_n).
\end{split}\end{equation}

The final step of the proof consists of removing the weight function $\omega$ on the right hand side of \eqref{PROP.1}. We denote $r(x) = \omega(\frac{2\pi x}{\log T}) g (x)$, a Lipschitz continuous integrable function. We want to show that $\hat r$ is approximately $\hat g$. By the convolution theorem and the well-known formula $\hat \omega(\alpha) = 2\pi e^{-4\pi|\alpha|}$, we have
\begin{equation}\begin{split}\notag
    \hat r(\alpha) = \bigg(\frac{ \hat \omega(\frac{\cdot \log T}{2\pi})}{2\pi/\log T} * \hat g \bigg)(\alpha)
    = \int_{\mathbb R} T^{-2|y|}\log T \, \hat g(\alpha-y) dy.
\end{split}\end{equation}
Since $\hat g$ is such that $\hat g'(\alpha)\ll \min\{1,|\alpha|^{-3}\}$, by an integration by parts we obtain
\begin{equation}\begin{split}\label{PROP.11}
    \hat r(\alpha) 
    = \hat g(\alpha) + O\bigg( \int_{\mathbb R} T^{-2|y|} |\hat g'(\alpha-y)|dy \bigg)
    = \hat g(\alpha) + O\bigg( \frac{1}{(1+|\alpha|^3)\log T}\bigg).
\end{split}\end{equation}
Equations \eqref{PROP.1}, \eqref{PROP.10}, and \eqref{PROP.11} yield the claim.

\section{Proof of Proposition \ref{TPC_SmallAlpha}}\label{SecSmallAlpha}

We first state an explicit formula, essentially due to Montgomery \cite{Montgomery1}.

\begin{lemma}\label{ExplicitFormulaRevisited}
    Assume RH. For $y\geq 1$ and $T\leq t \leq 2T$, we have
    \begin{equation}\begin{split}\notag   
        2\sum_{\gamma} \frac{y^{i\gamma}}{1+(t-\gamma)^2}
        = - \frac{1}{\sqrt{y}}  \sum_m \Lambda(m) a_m(y) \bigg(\frac{y}{m}\bigg)^{it} 
        &+ y^{-1+it} \bigg(\log\frac{t}{2\pi} + \frac{\zeta'}{\zeta}(3/2-it)\bigg) \\
        &+ O\bigg(\frac{\sqrt{y}}{T^2}\bigg) + O\bigg(\frac{1}{yT}\bigg)
    \end{split}\end{equation}
    where
    $$a_m(y)= \begin{cases} 
    (\frac{m}{y})^{1/2} & \text{if } m\leq y \\ (\frac{y}{m})^{3/2} & \text{if } m> y.\end{cases}$$
\end{lemma}

\begin{proof}
    We take $\sigma=\frac{3}{2}$ in \cite[Equation (22)]{Montgomery1} and bound trivially the last two terms for $T\leq t \leq 2T$: 
    \begin{equation}\begin{split}\notag   
        2\sum_{\gamma} \frac{y^{i\gamma}}{1+(t-\gamma)^2}
        = - \frac{1}{\sqrt{y}}  \sum_m \Lambda(m) a_m(y) \bigg(\frac{y}{m}\bigg)^{it} 
        &- y^{-1+it} \frac{\zeta'}{\zeta}(-1/2+it)\\
        &+ O\bigg(\frac{\sqrt{y}}{T^2}\bigg) + O\bigg(\frac{1}{y^2T}\bigg).
    \end{split}\end{equation}
    We now apply the functional equation in the form
    $$-\frac{\zeta'}{\zeta}(-\frac{1}{2}+it) 
    = \frac{\zeta'}{\zeta}\bigg(\frac{3}{2}-it\bigg) + \frac{1}{2}\frac{\Gamma'}{\Gamma}\bigg(-\frac{1}{4}+\frac{it}{2}\bigg) + \frac{1}{2} \frac{\Gamma'}{\Gamma}\bigg(\frac{3}{4}-\frac{it}{2}\bigg) -\log \pi.$$
    Since Stirling's formula yields
    $$\frac{1}{2}\frac{\Gamma'}{\Gamma}\bigg(-\frac{1}{4}+\frac{it}{2}\bigg) + \frac{1}{2} \frac{\Gamma'}{\Gamma}\bigg(\frac{3}{4}-\frac{it}{2}\bigg) = \frac{1}{2}\log\bigg(\frac{t^2}{4}\bigg) + O\bigg(\frac{1}{t}\bigg),$$
    the proof is concluded.
\end{proof}

We introduce the following notation; for any $\delta>0$, we denote 
$$\mathcal R = 
\frac{1}{\log T} 
+\frac{\sqrt{n}(\log T)^3}{T}
+ \frac{n^2}{T^2} 
+ \frac{n^{1+\delta}(\log T)^2}{T}.$$
Note that $\mathcal R\ll 1/\log T$, if $n\leq T^{1-\varepsilon}$ for some $\varepsilon>\delta$.

\subsection{Positive alpha}

We start by handling the case $0<\alpha<1-\frac{\log n}{\log T}-\delta_T$, where $\delta_T=10\log\log T/\log T$.
In this range, $nT^\alpha$ and $T^{\alpha}$ are both between 1 and $T$. So we can use classical methods to compute averages of short Dirichlet polynomials. 
We recall the integral expression for the function $\omega$, namely:
\begin{equation}\notag
\omega(\gamma-\gamma') = \frac{2}{\pi}\int_{\mathbb R} 
\frac{1}{[1+(t-\gamma)^2][1+(t-\gamma')^2]} dt. 
\end{equation}
Arguing like in \cite[p. 187-188]{Montgomery1} (in the other direction), one can extend the range of summation over zeros and restrict the range of integration in the definition of $F_n(\alpha)$ \eqref{DefinitionOfFn} at a cost of an acceptable error term, getting
\begin{equation}\notag
F_n(\alpha) 
= -\frac{\sqrt{n}}{T\Lambda(n)} \int_{T}^{2T} 
\bigg(2\sum_{\gamma} \frac{(nT^\alpha)^{i\gamma} }{1+(t-\gamma)^2}\bigg)\bigg( 2\sum_{\gamma'}\frac{(T^{-\alpha})^{i\gamma'}}{1+(t-\gamma')^2}\bigg) dt 
+ O\bigg(\frac{\sqrt{n}(\log T)^3}{T}\bigg).
\end{equation}
By applying Lemma \ref{ExplicitFormulaRevisited} for both sums over zeros, we write
\begin{equation}\label{7sep.1}
    F_n(\alpha) 
    = F_n^{1-1}(\alpha) + F_n^{1-2}(\alpha) + F_n^{1-2}(\alpha) + F_n^{2-2}(\alpha) + O(\mathcal R)
\end{equation}
with
\begin{equation}\begin{split}\notag
    F_n^{1-1}(\alpha) 
    &= - \frac{T^{-\alpha}}{\Lambda(n)} \frac{1}{T}\int_T^{2T}
    \sum_m \sum_l \Lambda(m) \Lambda(l) a_m(nT^\alpha) a_l(T^\alpha) \bigg(\frac{nl}{m}\bigg)^{it} dt \\
    F_n^{1-2}(\alpha) 
    &= \frac{T^{-3\alpha/2}}{\Lambda(n)} \frac{1}{T} \int_T^{2T}  \sum_m \Lambda(m) a_m(nT^\alpha) \bigg(\frac{n}{m}\bigg)^{it} \bigg( \log\frac{t}{2\pi} + \frac{\zeta'}{\zeta}(3/2+it)\bigg) dt\\
    F_n^{2-1}(\alpha) 
    &= \frac{T^{-3\alpha/2} }{\sqrt{n}\Lambda(n)} 
    \frac{1}{T} \int_T^{2T} \sum_l \Lambda(l) a_l(T^\alpha) (nl)^{it} \bigg( \log\frac{t}{2\pi} + \frac{\zeta'}{\zeta}(3/2-it)\bigg) dt\\
    F_n^{2-2}(\alpha) 
    &= -\frac{T^{-2\alpha}}{\sqrt{n}\Lambda(n)} \frac{1}{T} \int_T^{2T} n^{it} \bigg| \log\frac{t}{2\pi} + \frac{\zeta'}{\zeta}(3/2-it)\bigg|^2 dt.
\end{split}\end{equation}
To handle the error terms from Lemma \ref{ExplicitFormulaRevisited}, we used the
crude bound $\int_T^{2T}|\sum_m a_mm^{-it}|^2dt \ll \sum_m |a_m|^2(T+m)$
and the prime number theorem to obtain
\begin{equation}\begin{split}\notag
    \int_T^{2T} & \bigg| \sum_m \Lambda(m)a_m(y)\bigg(\frac{y}{m}\bigg)^{it} \bigg| dt 
    \ll \sqrt{T\sum_m \Lambda(m)^2a_m(y)^2(T+m)}
    \ll T\sqrt{y\log y} + y\sqrt{T\log y}. 
\end{split}\end{equation}

First we analyze $F_n^{1-1}(\alpha)$. For starters, we notice that if $nl\neq m$ then
$$\int_T^{2T} \bigg(\frac{nl}{m}\bigg)^{it} dt \ll \frac{1}{|\log\frac{nl}{m}|} \ll \frac{m}{|nl-m|}.$$
Therefore, the \lq\lq off-diagonal\rq\rq contribution of $F_n^{1-1}(\alpha)$ is
$$ \ll \frac{T^{-\alpha}}{\Lambda(n)} \frac{1}{T} \sum_{\substack{m,l \\ m\neq ln}}  \Lambda(m) \Lambda(l) a_m(nT^\alpha) a_l(T^\alpha) \frac{m}{|nl-m|} 
\ll \frac{nT^\alpha(\log T)^3}{T},$$
since
\begin{equation}\label{ClaimDoubleSum}
\sum_{\substack{m,l \\ m\neq ln}}  \Lambda(m) \Lambda(l) a_m(nT^\alpha) a_l(T^\alpha) \frac{m}{|nl-m|}  \ll (\log T)^3 nT^{2\alpha} .
\end{equation}
The proof of \eqref{ClaimDoubleSum} is standard. One splits the double sum into four terms, according to the definition of $a_m(nT^\alpha)$ and $a_l(T^\alpha)$. For each piece, one splits again into two cases: $0<|nl-m|<\frac{nl}{2}$ and $|nl-m|>\frac{nl}{2}$. We show all the details for one of the four terms, as the other three can be handled analogously. For (say) $m\leq nT^\alpha$ and $l\leq T^\alpha$, the left hand side of \eqref{ClaimDoubleSum} is bounded by
$$\ll \frac{(\log T)^2}{\sqrt{n}T^\alpha}\sum_{\substack{l\leq T^\alpha \\ m\leq nT^\alpha \\ m\neq ln}} \frac{m\sqrt{ml}}{|nl-m|}.$$
We now split into two cases. If $|nl-m|=:|h|< \frac{nl}{2}$, then the sum above is
$$\sum_{\substack{l\leq T^\alpha \\ m\leq nT^\alpha \\ 0<|nl-m|<\frac{nl}{2}}} \frac{m^{3/2}l^{1/2}}{|nl-m|} 
\ll \sum_{\substack{l\leq T^\alpha \\ 0<|h|<\frac{nl}{2}}} \frac{(nl\pm h)^{3/2}l^{1/2}}{|h|} 
\ll n^{3/2} \sum_{\substack{l\leq T^\alpha \\ h<\frac{nl}{2}}} \frac{l^2}{|h|} 
\ll n^{3/2} T^{3\alpha} \log T. $$
If instead $|nl-m|> \frac{nl}{2}$ then $|nl-m|\gg nl$, so
$$\sum_{\substack{l\leq T^\alpha \\ m\leq nT^\alpha \\|nl-m|>\frac{nl}{2}}} \frac{m^{3/2}l^{1/2}}{|nl-m|} 
\ll \frac{1}{n}\sum_{\substack{l\leq T^\alpha \\ m\leq nT^\alpha}} \frac{m^{1/2}}{l^{1/2}} 
\ll \frac{(nT^\alpha)^{5/2} (T^{\alpha})^{1/2}}{n}
\ll n^{3/2}T^{3\alpha}, $$
and \eqref{ClaimDoubleSum} is proven in this case.
Hence, since $0<\alpha<1-\frac{\log n}{\log T}-\delta_T$,
\begin{equation}\begin{split}\label{7sep.2}
    F_n^{1-1}(\alpha) 
    &= - \frac{T^{-\alpha}}{\Lambda(n)} 
    \sum_l \Lambda(nl) \Lambda(l) a_{nl}(nT^\alpha) a_l(T^\alpha)  + O(T^{-\delta_T/2})\\
    &= - \frac{1}{\Lambda(n)} 
    \sum_{m} \frac{\Lambda(nm) \Lambda(m)}{m} \min\bigg\{\frac{m}{T^\alpha},\frac{T^\alpha}{m}\bigg\}^2
    + O(T^{-\delta_T/2}).
\end{split}\end{equation}

Now we move to $F_n^{1-2}(\alpha)$. We isolate the term $m=n$, for which $a_n(nT^\alpha)=T^{-\alpha/2}$. Since $\log\frac{t}{2\pi} + \frac{\zeta'}{\zeta}(3/2+it)=\log T+O(1)$ for $T\leq t \leq 2T$, we get
\begin{equation}\begin{split}\label{7sep.3}
    F_n^{1-2}(\alpha) 
    = &T^{-2\alpha}(\log T+O(1)) 
    + O(\mathcal E_1)
    + O(\mathcal E_2)
\end{split}\end{equation}
with
\begin{equation}\notag
    \mathcal E_1  
    \ll \frac{T^{-3\alpha/2}}{\Lambda(n)} \frac{1}{T} \sum_{m\neq n} \Lambda(m) a_m(nT^\alpha) \bigg|\int_T^{2T} \bigg(\frac{n}{m}\bigg)^{it} \log\frac{t}{2\pi} dt\bigg|
\end{equation}
and
$$\mathcal E_2 \ll \frac{T^{-3\alpha/2}}{\Lambda(n)} \frac{1}{T} \sum_{m} \Lambda(m) a_m(nT^\alpha) \bigg|\int_T^{2T} \bigg(\frac{n}{m}\bigg)^{it} \frac{\zeta'}{\zeta}(3/2+it) dt\bigg|.$$
We start by bounding $\mathcal E_1$. Using the first derivative test for oscillating integrals (see e.g. \cite[Lemma 4.3]{Titchmarsh1}), we obtain
\begin{equation}\notag
    \mathcal E_1 
    \ll \frac{T^{-3\alpha/2}}{\Lambda(n)} \frac{\log T}{T} \sum_{m\neq n}  \frac{\Lambda(m) a_m(nT^\alpha)}{|\log(m/n)|}
    \ll \frac{T^{-3\alpha/2}}{\Lambda(n)} \frac{\log T}{T} \sum_{m\neq n}  \frac{m\Lambda(m) a_m(nT^\alpha)}{|m-n|}.
\end{equation}
The remaining sum can be bounded with a similar strategy as in the proof of \eqref{ClaimDoubleSum}. Namely, one splits in two cases: if $m\geq2n$ then $|m-n|\gg m$, so the sum is clearly $\ll nT^\alpha$. If instead $m< 2n$, we do the change of variable $h=m-n$, and we get a contribution bounded by $\ll nT^\alpha(\log n)^2$, because the sum over $h$ is short (i.e. $0<|h|<n$). Arguing as described, one gets
\begin{equation}\notag
    \mathcal E_1 
    \ll \frac{nT^{-\alpha/2}\log T(\log n)^2}{\Lambda(n)T}
    \ll \frac{n(\log n)^2\log T}{T}
    \ll \mathcal R.
\end{equation}
Now we deal with $\mathcal E_2$.
To bound the integral, we open the log-derivative of zeta as a Dirichlet series and integrate over $t$ term by term, getting:
\begin{equation}\notag
\int_T^{2T} \bigg(\frac{n}{m}\bigg)^{it} \frac{\zeta'}{\zeta}(3/2+it) dt
\ll \frac{\Lambda(n/m)T}{(n/m)^{3/2}}
+ \sum_{bm\neq n} \frac{\Lambda(b)}{b^{3/2}}\frac{m}{|bm-n|} .
\end{equation}
Therefore, 
\begin{equation}\notag\label{4july.1}
\mathcal E_2 
\ll \frac{T^{-3\alpha/2}}{\Lambda(n)} \sum_{m}   \frac{\Lambda(m) \Lambda(n/m) a_m(nT^\alpha)}{(n/m)^{3/2}}
+ \frac{T^{-3\alpha/2}}{T} \sum_{m} \Lambda(m) a_m(nT^\alpha) \sum_{bm\neq n} \frac{\Lambda(b)}{b^{3/2}}\frac{n}{|bm-n|}.
\end{equation}
The first term above can be bounded easily. Since $n=q^a$, the factor $\Lambda(n/m)$ forces $m$ to be a power of $q$ smaller than $n$. Then, the sum over $m$ can be restricted to the range $m<n$, in which case $a_m(nT^\alpha)= \sqrt{m}/\sqrt{nT^\alpha}$. Hence, the first term is
$$\ll \frac{T^{-2\alpha}}{\Lambda(n)n^2} \sum_{m<n} \Lambda(m) \Lambda(n/m) m^2
\ll \frac{T^{-2\alpha}\Lambda(n)}{n^2} \sum_{c<a} q^{2c}
\ll \frac{T^{-2\alpha}\Lambda(n)}{n^2} \frac{n^2}{q^2}
\ll T^{-2\alpha}.$$
For the second term, one can apply our now familiar machinery for this kind of sums, and show that it is $\ll \mathcal R$. The proof works exactly as the similar ones described above: if $|bm-n|$ is (say) $\geq \frac{n}{2}$, then the sum is clearly smaller than $n/T$. If instead $|bm-n|<\frac{n}{2}$, we do the change of variable $h=bm-n$ and we still win because the sum over $h$ is short.
This yields 
$$\mathcal E_2 \ll T^{-2\alpha} + \mathcal R .$$
Therefore \eqref{7sep.3} reads
\begin{equation}\label{7sep.4}
    F_n^{1-2}(\alpha) 
    = T^{-2\alpha}(\log T+O(1))
    + O(\mathcal R).
\end{equation}

Arguing similarly, we can show that
\begin{equation}\label{8sep.1}
F_n^{2-1}(\alpha) \ll T^{-2\alpha} + \mathcal R.
\end{equation}
The computation follows closely that of $F_n^{1-2}(\alpha)$. The main difference is the phase of the oscillating integral, which is now $(nl)^{it}$. Since $n$ and $l$ are prime powers, the phase is never close to 1, and therefore the integral has no diagonal contribution, i.e. no main term.  

Also $F_n^{2-2}(\alpha)$ can be handled with similar techniques. For starters, we expand
$$ \bigg| \log\frac{t}{2\pi} + \frac{\zeta'}{\zeta}(3/2-it)\bigg|^2
= \bigg(\log\frac{t}{2\pi}\bigg)^2 
+ \log\frac{t}{2\pi}\sum_b \frac{\Lambda(b)}{b^{3/2}}(b^{it}+b^{-it})
+ O(1).$$
The contribution of the error term above to $F_n^{2-2}(\alpha)$ is clearly $\ll T^{-2\alpha}$. Moreover, the term $(\log\frac{t}{2\pi})^2$ contributes $\ll (\log T)^2/T$, by an application of the first derivative test \cite[Lemma 4.3]{Titchmarsh1}. The same lemma also provides a bound for the term involving $b^{it}$, which turns out to be $\ll \log T/T$.
Therefore,
\begin{equation}\begin{split}\notag
    F_n^{2-2}(\alpha) 
    &= \frac{T^{-2\alpha}}{\sqrt{n}\Lambda(n)} \frac{1}{T} \sum_b \frac{\Lambda(b)}{b^{3/2}}  \int_T^{2T} \bigg(\frac{n}{b}\bigg)^{it} \log\frac{t}{2\pi} dt
    + O(T^{-2\alpha}) +O \bigg(\frac{(\log T)^2}{T}\bigg ).
\end{split}\end{equation} 
As usual, we isolate the main term coming from the diagonal term $n=b$, getting
\begin{equation}\begin{split}\label{8sep.2}
    F_n^{2-2}(\alpha) 
    &=\frac{T^{-2\alpha}\log T}{n^2} 
    +O\bigg( \frac{T^{-2\alpha}}{\sqrt{n}T} \sum_{b\neq n} \frac{\Lambda(b)}{b^{3/2}} \frac{\log T}{|\log(n/b)|} \bigg) +O(T^{-2\alpha}) 
    +O \bigg(\frac{(\log T)^2}{T}\bigg )\\
    &=\frac{T^{-2\alpha}\log T}{n^2} 
    +O(\mathcal R) +O(T^{-2\alpha})
\end{split}\end{equation}
since
$$ \frac{T^{-2\alpha}}{\sqrt{n}T} \sum_{b\neq n} \frac{\Lambda(b)}{b^{3/2}} \frac{\log T}{|\log(n/b)|}
\ll \frac{\log T}{\sqrt{n}T} \sum_{b\neq n} \frac{\Lambda(b)}{b^{3/2}} \frac{n}{|n-b|}
\ll \frac{\sqrt{n}\log T}{T} 
\ll \mathcal R .$$ 

Plugging \eqref{7sep.2}, \eqref{7sep.4}, \eqref{8sep.1}, and \eqref{8sep.2} into \eqref{7sep.1}, we conclude the proof.

\subsection{Negative alpha} 

The case $-\frac{\log n}{\log T}\leq\alpha\leq0$ is similar to the case of positive $\alpha$. In this range $1\leq T^{-\alpha}\leq n$ and $1\leq nT^{\alpha}\leq n$. Therefore, arguing as in \cite[p. 187-188]{Montgomery1} and applying Lemma \ref{ExplicitFormulaRevisited}, we have
\begin{equation}\begin{split}\label{negative.1}
    F_n(\alpha) 
    &= -\frac{4\sqrt{n}}{T\Lambda(n)} \int_{T}^{2T} 
    \bigg(\sum_{\gamma} \frac{(nT^\alpha)^{i\gamma} }{1+(t-\gamma)^2}\bigg)\bigg( \sum_{\gamma'}\frac{(T^{-\alpha})^{i\gamma'}}{1+(t-\gamma')^2}\bigg) dt + O\bigg(\frac{\sqrt{n}(\log T)^3}{T}\bigg) \\
    &= F_n^{1-1}(\alpha) + F_n^{1-2}(\alpha) + F_n^{2-1}(\alpha) + F_n^{2-2}(\alpha) + O\bigg(\frac{\sqrt{n}(\log T)^3}{T}\bigg)
\end{split}\end{equation}
where
\begin{equation}\begin{split}\notag
    F_n^{1-1}(\alpha) 
    &= -\frac{1}{T\Lambda(n)} \sum_m \sum_l \Lambda(m) \Lambda(l) a_m(nT^\alpha) a_l(T^{-\alpha}) \int_{T}^{2T} 
    \bigg(\frac{n}{ml}\bigg)^{it} dt \\
    F_n^{1-2}(\alpha) 
    &= \frac{T^{\alpha/2}}{T\Lambda(n)} \sum_m \Lambda(m) a_m(nT^\alpha) \int_{T}^{2T} 
     \bigg(\frac{n}{m}\bigg)^{it}
    \bigg(\log\frac{t}{2\pi} + \frac{\zeta'}{\zeta}(3/2-it)\bigg) dt\\
    F_n^{2-1}(\alpha) 
    &= \frac{1}{T\sqrt{nT^\alpha}\Lambda(n)} \sum_l \Lambda(l) a_l(T^{-\alpha}) \int_{T}^{2T} \bigg(\log\frac{t}{2\pi} + \frac{\zeta'}{\zeta}(3/2-it)\bigg)  \bigg(\frac{n}{l}\bigg)^{it} dt\\
    F_n^{2-2}(\alpha) 
    &= -\frac{1}{T\sqrt{n}\Lambda(n)} \int_{T}^{2T} n^{it}  \bigg(\log\frac{t}{2\pi} + \frac{\zeta'}{\zeta}(3/2-it)\bigg)^2 dt.
\end{split}\end{equation}

As in the previous proof, we now analyze the four terms one by one. We start by $F_n^{1-1}$; we isolate the contribution of the diagonal $ml=n$, for which we use the trivial identity $a_m(nT^\alpha)a_{n/m}(T^{-\alpha}) = \min\{\frac{nT^\alpha}{m},\frac{m}{{nT^\alpha}}\}^2$. 
The off-diagonal terms $ml\neq n$ can be bounded by $\ll (\log T)^2n^{1+\delta}/T$ for any $\delta>0$, in a similar (and easier) way as in \eqref{ClaimDoubleSum}. 
For example, when $m\leq nT^\alpha$ and $l\leq T^{-\alpha}$, the off-diagonal contribution can be bounded by
\begin{equation}\begin{split}\notag
\frac{(\log T)^2}{\sqrt{n}T}\sum_{\substack{m\leq nT^\alpha \\ l\leq T^{-\alpha}}} \frac{(ml)^{3/2}}{n-ml}
&\ll \frac{(\log T)^2}{\sqrt{n}T}\sum_{ml\leq \frac{n}{2}} \frac{(ml)^{3/2}}{n-ml}
+ \frac{(\log T)^2}{\sqrt{n}T}\sum_{\frac{n}{2}<ml< n} \frac{(ml)^{3/2}}{n-ml}\\
&\ll \frac{(\log T)^2}{n\sqrt{n}T}\sum_{ml\leq \frac{n}{2}} (ml)^{3/2}
+ \frac{n(\log T)^2}{T}\sum_{0<h<\frac{n}{2}} \frac{d(n-h)}{h}
\ll \frac{n^{1+\delta}(\log T)^2}{T}.
\end{split}\end{equation}
Hence, we obtain
\begin{equation}\label{negative.2}
    F_n^{1-1}(\alpha) 
    = -\frac{1}{\Lambda(n)} \sum_m \Lambda(m)\Lambda(n/m) \min\bigg\{ \frac{nT^\alpha}{m},\frac{m}{{nT^\alpha}} \bigg\}^2
    +O(\mathcal R).
\end{equation}

To analyze $F_n^{1-2}(\alpha)$, we apply the same machinery as in the previous section. For the $\log\frac{t}{2\pi}$ term, we isolate the diagonal term $n=m$, for which $a_n(nT^\alpha)=T^{3\alpha/2}$. As for the $\frac{\zeta'}{\zeta}(3/2-it)$-term, we expand the log-derivative of zeta as a Dirichlet series indexed by the parameter $b$, and we isolate the contribution from $b=m/n$. This yields
\begin{equation}\begin{split}\notag
    F_n^{1-2}(\alpha) 
    &= T^{2\alpha}(\log T+O(1)) 
    - \frac{T^{\alpha/2}n^{3/2}}{\Lambda(n)} \sum_m \frac{\Lambda(m) \Lambda(m/n) a_m(nT^\alpha)}{m^{3/2}}  \\
    &\quad +O\bigg( \frac{T^{\alpha/2}}{T} \sum_{m \neq n} \Lambda(m) a_m(nT^\alpha) \frac{n\log T}{|n-m|} \bigg)\\
    &\quad + O\bigg( \frac{T^{\alpha/2}}{T} \sum_{bn\neq m} \frac{ \Lambda(m)\Lambda(b)}{b^{3/2}} a_m(nT^\alpha) \frac{m}{|bn-m|} \bigg).
\end{split}\end{equation} 
Bounding the error terms above by $\ll n\log T(\log n)^2/T \ll \mathcal R$ is now a routine calculation. 
Moreover, the second term on the first line in the display above is also small. Indeed, since $n=q^a$, the sum is supported on values of $m$ such that $m$ is a power of $q$, and $m>n\geq nT^\alpha$. For these $m$, we have $a_m(nT^\alpha) = (nT^\alpha/m)^{3/2}$, and we therefore obtain
$$\frac{T^{\alpha/2}n^{3/2}}{\Lambda(n)} \sum_m \frac{\Lambda(m) \Lambda(m/n) a_m(nT^\alpha)}{m^{3/2}} 
= \frac{T^{2\alpha}n^3}{\Lambda(n)} \sum_{c>a} \frac{\Lambda(n)^2}{q^{3c}} 
\ll T^{2\alpha} \frac{\log q}{q^3}
\ll T^{2\alpha}.$$
As a consequence,
\begin{equation}\begin{split}\label{negative.3}
    F_n^{1-2}(\alpha) 
    &= T^{2\alpha}(\log T+O(1)) 
    +O(\mathcal R).
\end{split}\end{equation} 

An analogous argument leads to
\begin{equation}\begin{split}\label{negative.4}
    F_n^{2-1}(\alpha) 
    =\frac{\log T+O(1)}{(nT^\alpha)^2} 
    +O(\mathcal R).
\end{split}\end{equation}

Finally, we handle $F_n^{2-2}(\alpha)$ with similar techniques as in the previous section. We first expand the square 
$$\bigg(\log\frac{t}{2\pi} + \frac{\zeta'}{\zeta}(3/2-it)\bigg)^2
= \bigg(\log\frac{t}{2\pi}\bigg)^2
- 2\log\frac{t}{2\pi} \sum_b\frac{\Lambda(b)}{b^{3/2-it}}
+ \sum_{b,c} \frac{\Lambda(b)\Lambda(c)}{(bc)^{3/2-it}}.$$
We plug the above equation in the definition of $F_n^{2-2}(\alpha)$. The contribution from the term $(\log\frac{t}{2\pi})^2$ to $F_n^{2-2}(\alpha)$ is certainly $\ll (\log T)^2/T$, by a standard application of the first derivative test \cite[Lemma 4.3]{Titchmarsh1}. The second and third term can be bounded similarly, since the oscillating terms ($(nb)^{it}$ and $(nbc)^{it}$ respectively) are never close to 1. Doing so, we obtain 
\begin{equation}\label{negative.5}
    F_n^{2-2}(\alpha) 
    \ll \mathcal R.
\end{equation}
The desired claim follows from Equations \eqref{negative.1}-\eqref{negative.5}

\section{Proof of Proposition \ref{TPC_IntermediateAlpha}}\label{SecIntermAlpha}

Following \cite{GG}, we introduce some convenient notations:
\begin{equation}\begin{split}
    A(s) := \sum_{m \le N} \frac{\Lambda(m)}{m^s}, \qquad 
    B(s) := \sum_{m \le N} \frac{\Lambda(m/n)}{m^s}, \\
    A^*(s) := \sum_{m > N} \frac{\Lambda(m)}{m^s}, \qquad 
    B^*(s) := \sum_{m > N} \frac{\Lambda(m/n)}{m^s}.
\end{split}\end{equation}
and
\begin{align*}
    \mathcal{A}(t) &:= \frac{1}{N} \left( A \left( - \frac{1}{2} + i t \right) - \int_1^N u^{1/2- i t} \, du \right), \\ 
    \mathcal{A}^*(t) &:= N \left( A^* \left(\frac{3}{2} + i t \right) - \int_N^\infty u^{-3/2 - it} \, du \right), \\
    \mathcal{B}(t) &:= \frac{1}{N} \left( B \left( - \frac{1}{2} + i t \right) - \frac{1}{n} \int_1^N u^{1/2- i t} \, du\right), \\
    \mathcal{B}^*(t) &:= N \left( B^* \left(\frac{3}{2} + i t \right) - \frac{1}{n} \int_N^\infty u^{-3/2 - it} \, du \right).
\end{align*}

The first lemma we prove is an application of the explicit formula. For convenience, we introduce the notation
\begin{equation}\begin{split}\label{defErrorE}
    \mathcal S
    := \frac{n\log T}{NT} \int_{\mathbb R} \bigg( \sqrt{n}|\mathcal A(t)| + \sqrt{n}|\mathcal A^*(t)|  + |\mathcal B(t)| &+ |\mathcal B^*(t)| \bigg) \psi_U\bigg(\frac{t}{T}\bigg) dt\\
    &+\frac{\sqrt{n}(\log T)^3}{T} 
    + \frac{n^{3/2}(\log T)^2}{N^2},
\end{split}\end{equation} 
which will appear as an error term in the following.

\begin{lemma}\label{IntermediateRangeExplForm}
    Assume RH. Let $n$ be a prime power and denote $N=nT^\alpha$. Also, suppose that $\alpha>0$. Then we have
    \begin{equation}
        F_n(\alpha;\psi_U) = - \frac{n}{T\Lambda(n)} \int_{\mathbb R} \left( \mathcal{A}(t) + \mathcal{A}^*(t) \right) \overline{\left( \mathcal{B}(t) + \mathcal{B}^*(t) \right)} \psi_U\bigg(\frac{t}{T}\bigg) dt 
        + O(\mathcal S),
    \end{equation}
    where $\mathcal S$ is defined in \eqref{defErrorE}.
    
\end{lemma}

\begin{proof}
By definition of $\omega_{\psi_U}$ \eqref{SmoothOmega}, we have
\begin{equation}\begin{split}\notag
    F_n(\alpha;\psi_U) 
    &= - \frac{\sqrt{n}}{T\Lambda(n)}  
    \int_{\mathbb R} \left( 2 \sum_{T\leq \gamma\leq 2T} \frac{(nT^\alpha)^{i\gamma}}{1+(t-\gamma)^2} \right) \left( 2 \sum_{T\leq \gamma'\leq 2T} \frac{(T^\alpha)^{-i\gamma'}}{1+(t-\gamma')^2} \right) \psi_U\bigg(\frac{t}{T}\bigg) dt\\
    &= - \frac{\sqrt{n}}{T\Lambda(n)}  
    \int_{\mathbb R} \left( 2 \sum_{\gamma} \frac{(nT^\alpha)^{i\gamma}}{1+(t-\gamma)^2} \right) \left( 2 \sum_{\gamma'} \frac{(T^\alpha)^{-i\gamma'}}{1+(t-\gamma')^2} \right) \psi_U\bigg(\frac{t}{T}\bigg) dt
    + O(\mathcal S).
\end{split}\end{equation}
The second line is obtained by  arguing like in \cite[p. 187-188]{Montgomery1}, since $\psi_U$ is a smooth function supported on $[1,2]$. It suffices to note that the conditions restricting the sums over zeros can be removed up to an error term $\ll \sqrt{n}(\log T)^3/T \ll \mathcal S$.
We now employ a formula due to Montgomery \cite{Montgomery1} in the following version (see e.g. \cite[(4.5)]{GGOS}):
\begin{equation}\begin{split}\notag
   -2 \sum_{\gamma} \frac{x^{i\gamma}}{1+(t-\gamma)^2} 
   = &\frac{x^{it}}{x}\bigg( \sum_{m\leq x} \frac{\Lambda(m)}{m^{-1/2+it}} - \frac{x^{3/2-it}}{\frac{3}{2}-it} \bigg)\\
   &+x^{1+it}\bigg( \sum_{m> x} \frac{\Lambda(m)}{m^{3/2+it}} - \frac{x^{-1/2-it}}{\frac{1}{2}+it} \bigg) + O\bigg(\frac{\log(|t|+2)}{x}\bigg)
\end{split}\end{equation}
for $x\geq 1$. In the notations we introduced in the beginning of the section, by a double application of the above formula and some trivial manipulations, we obtain
\begin{equation}\begin{split}\notag
    F_n(\alpha;\psi_U) 
    &= - \frac{n}{T\Lambda(n)}  
    \int_{\mathbb R} \left( \mathcal A(t) + \mathcal A^*(t) + O\bigg(\frac{\log T}{N}\bigg) \right)\\
    &\hspace{3cm} \times\left( \overline{\mathcal B(t)} + \overline{\mathcal B^*(t)} + O\bigg(\frac{\sqrt{n}\log T}{N}\bigg) \right) \psi_U\bigg(\frac{t}{T}\bigg) dt 
    + O (\mathcal S)\\
    &= - \frac{n}{T\Lambda(n)}  
    \int_{\mathbb R} \left( \mathcal A(t) + \mathcal A^*(t)  \right)
    ( \overline{\mathcal B(t)} + \overline{\mathcal B^*(t)}  ) \psi_U\bigg(\frac{t}{T}\bigg) dt 
    +O(\mathcal S_1) + (\mathcal S),
\end{split}\end{equation}
where
\begin{equation}\begin{split}\notag
    \mathcal S_1 \ll \frac{n\log T}{NT} \int_{\mathbb R} \bigg( \sqrt{n}|\mathcal A(t)| + \sqrt{n}|\mathcal A^*(t)|  +|\mathcal B(t)| + |\mathcal B^*(t)|\bigg) \psi_U\bigg(\frac{t}{T}\bigg) dt +O\bigg(\frac{n^{3/2}(\log T)^2}{N^2}\bigg).
\end{split}\end{equation}
\end{proof}

\begin{lemma}\label{LemmaHLC}
    Assume RH and Conjecture \ref{HLC}. Let $n$ be a prime power, and denote $N:=nT^\alpha$. Suppose that $T\leq N\leq T^2$.
    Then, for any $\nu\in(0,\frac{1}{2})$, we have
    \begin{equation}\begin{split}\notag
        F_n(\alpha;\psi_U) = &-\frac{2}{\Lambda(n)} \int_1^\infty f_n(y) \Re\hat\psi_U\left( \frac{y T}{2 \pi N} \right) \, dy \\
        &+ O\bigg(\frac{1}{\Lambda(n)}\bigg) 
        +O\bigg( \frac{nN^{1+5\nu}}{T^2} + \frac{nN^{\frac{1}{2}+\frac{11}{2}\nu}}{T} + nT^{-\frac{\nu}{2}} \bigg)
        + O(\mathcal S),
    \end{split}\end{equation}
    where $\mathcal S$ is defined in \eqref{defErrorE}.
\end{lemma}

\begin{proof}
    According to Lemma \ref{IntermediateRangeExplForm}, we write
    \begin{equation}\label{25sep.0}
        F_n(\alpha;\psi_U) 
        = F_n^{\mathcal A\overline{\mathcal B}}(\alpha;\psi_U) + F_n^{\mathcal A\overline{\mathcal B^*}}(\alpha;\psi_U) + F_n^{\mathcal A^*\overline{\mathcal B}}(\alpha;\psi_U) + F_n^{\mathcal A^* \overline{\mathcal B^*}}(\alpha;\psi_U)
        + O(\mathcal S), 
    \end{equation}
    with
    \begin{equation}\begin{split}\notag
        &F_n^{\mathcal A\overline{\mathcal B}}(\alpha;\psi_U)
        = -\frac{n}{T\Lambda(n)} \int_{\mathbb R} \mathcal{A}(t) \overline{\mathcal{B}(t)}  \psi_U\bigg(\frac{t}{T}\bigg) dt \\ 
        &F_n^{\mathcal A\overline{\mathcal B^*}}(\alpha;\psi_U)
        = -\frac{n}{T\Lambda(n)} \int_{\mathbb R} \mathcal{A}(t) \overline{\mathcal{B}^*(t)}  \psi_U\bigg(\frac{t}{T}\bigg) dt \\
        &F_n^{\mathcal A^* \overline{\mathcal B}}(\alpha;\psi_U)  
        = -\frac{n}{T\Lambda(n)} \int_{\mathbb R} \mathcal{A}^*(t) \overline{\mathcal{B}(t)}  \psi_U\bigg(\frac{t}{T}\bigg) dt \\
        &F_n^{\mathcal A^*\overline{\mathcal B^*}}(\alpha;\psi_U) 
        = -\frac{n}{T\Lambda(n)} \int_{\mathbb R} \mathcal{A}^*(t) \overline{\mathcal{B}^*(t)} \psi_U\bigg(\frac{t}{T}\bigg) dt .
    \end{split}\end{equation}
    In the range of $\alpha$ we are considering, we have $T<N<T^2$. Therefore, the above quantities involve mean values of \textit{long Dirichlet polynomial}. To evaluate them, we appeal to work of Goldston and Gonek \cite{GG}. 
    In order to apply their results, we note that
    \begin{equation}\begin{split}\notag
        \sum_{m\leq x} \Lambda(m) &= x + O(x^{1/2+\varepsilon}) \\
        \sum_{m\leq x} \Lambda(m/n) &= \frac{x}{n} + O\bigg(\bigg(\frac{x}{n}\bigg)^{1/2+\varepsilon}\bigg) 
    \end{split}\end{equation}
    by RH. Moreover, by the assumption of Conjecture \ref{HLC}, uniformly for $h\leq x^{1-\varepsilon}$ and for any $n$, we have 
    \begin{equation}\begin{split}\notag
        \sum_{m\leq x} \Lambda(m)\Lambda\bigg(\frac{m+h}{n}\bigg) 
        &= \sum_{h+1\leq m \leq x+h} \Lambda\bigg(\frac{m}{n}\bigg) \Lambda(m-h)
        = \mathfrak{S}_n(h) \frac{x}{n} + O(x^{1/2+\varepsilon}) \\
        \sum_{m\leq x} \Lambda\bigg(\frac{m}{n}\bigg) \Lambda(m+h)
        &= \mathfrak{S}_n(h)\frac{x}{n} + O(x^{1/2+\varepsilon}). 
    \end{split}\end{equation}
    We remark that the above formulas follow from Conjecture \ref{HLC} if (say) $n\leq \sqrt{x}$. Moreover, they become trivial in the complementary range $n>\sqrt{x}$. Hence, the two equations above do hold uniformly for every $n$.
    
    We start by looking at $F_n^{\mathcal A\overline{\mathcal B}}(\alpha;\psi_U)$. For any $0<\nu<\frac{1}{2}$, an application of \cite[Corollary 1]{GG} yields
    \begin{equation}\begin{split}\label{25sep.1}
        F_n^{\mathcal A\overline{\mathcal B}}(\alpha;\psi_U)
        &= -\frac{T^{-2\alpha}}{\Lambda(n)} \hat\psi(0)\sum_{l\leq T^\alpha} \Lambda(l)\Lambda(ln)l\\
        &\quad  -\frac{2}{\Lambda(n)}\frac{T^2}{(2\pi N)^2} \bigg( \int_{T/2\pi N}^{\infty} \sum_{1\leq h\leq 2\pi Nv/T} \mathfrak{S}_n(h)h^2 \times \Re\hat\psi_U(v) \frac{dv}{v^3}\\
        &\quad - \int_{T/2\pi\tau N}^{\infty} \frac{1}{3}\bigg(\frac{2\pi Nv}{T}\bigg)^3 \times \Re\hat\psi_U(v) \frac{dv}{v^3} \bigg)
        + O\bigg(\frac{nN^{1+5\nu}}{T^2}+\frac{nN^{\frac{1}{2}+\frac{11}{2}\nu}}{T}\bigg),
    \end{split}\end{equation}
    where $\tau= T^{1-\nu}$.
    Moreover, since $n=q^a$ is a prime power, we have
    $$ -\frac{T^{-2\alpha}}{\Lambda(n)} \sum_{l\leq T^\alpha} \Lambda(l)\Lambda(ln)l 
    =- \frac{\Lambda(n)}{T^{2\alpha}} \sum_{b\leq \alpha\frac{\log T}{\log q}} q^b
    =- \frac{\Lambda(n)}{T^{2\alpha}} \frac{q(T^\alpha-1)}{q-1}
    \ll \frac{\Lambda(n)}{T^\alpha} \ll \frac{1}{\sqrt{T}}. $$
    Finally, we can restrict the integral on the last line of \eqref{25sep.1} to the interval $(T/2\pi N,\infty)$, at the cost of an error bounded by $\ll (N/T)^2$.
    Therefore, with the change of variable $y=2\pi Nv/T$, Equation \eqref{25sep.1} reads
    \begin{equation}\begin{split}\label{25sep.2}
        F_n^{\mathcal A\overline{\mathcal B}}(\alpha;\psi_U)
        &= -\frac{2}{\Lambda(n)} \int_{1}^{\infty} \bigg(\sum_{1\leq h\leq y} \mathfrak{S}_n(h)h^2 - \frac{y^3}{3}\bigg)\times \Re\hat\psi\bigg(\frac{yT}{2\pi N}\bigg) \frac{dy}{y^3} \\
        &\hspace{3cm} + O\bigg(\frac{1}{\Lambda(n)}\bigg) 
        +O\bigg(\frac{nN^{1+5\nu}}{T^2}\bigg) + O\bigg(\frac{nN^{\frac{1}{2}+\frac{11}{2}\nu}}{T}\bigg).
    \end{split}\end{equation}
    We bound $F_n^{\mathcal A\overline{\mathcal B^*}}(\alpha;\psi_U)$ and $F_n^{\mathcal A^* \overline{\mathcal B}}(\alpha;\psi_U)$
    by a direct application of \cite[Theorem 3]{GG}; namely,
    \begin{equation}\begin{split}\label{25sep.3}
        F_n^{\mathcal A\overline{\mathcal B^*}}(\alpha;\psi_U) 
        \ll \frac{n}{T\Lambda(n)}\frac{N^{2-\frac{3}{2}+\frac{1}{2}+5\nu}}{T}
        \ll \frac{nN^{1+5\nu}}{T^2} 
        \quad \text{ and } \quad
        F_n^{\mathcal A^* \overline{\mathcal B}}(\alpha;\psi_U) \ll \frac{nN^{1+5\nu}}{T^2} .
    \end{split}\end{equation}
    Finally, we turn to $F_n^{\mathcal A^*\overline{\mathcal B^*}}(\alpha;\psi_U)$.  Applying Corollary 2 of \cite{GG}, we have 
    \begin{equation}\begin{split}\notag
        F_n^{\mathcal A^*\overline{\mathcal B^*}}(\alpha;\psi_U) 
        = &-\frac{nN^2}{T\Lambda(n)} \bigg\{ 
        \hat\psi_U(0)T \sum_{m>N} \frac{\Lambda(m)\Lambda(m/n)}{m^3} \\
        &+ 2\frac{(2\pi)^2}{nT} \int_{0}^{T/2\pi N} \sum_{1\leq h\leq H^*} \frac{\mathfrak{S}_n(h)}{h^2} \times \Re\hat\psi_U(v)v\,dv \\
        & + 2\frac{(2\pi)^2}{nT} \int_{T/2\pi N} ^{TH^*/2\pi N}\sum_{\frac{2\pi N v}{T}< h\leq H^*} \frac{\mathfrak{S}_n(h)}{h^2} \times\Re\hat\psi_U(v)v\,dv \\
        & - 2\frac{(2\pi)^2}{nT} \int_{0} ^{TH^*/2\pi N} \int_{2\pi Nv/T}^{H^*}\frac{dx}{x^2} \times \Re\hat\psi_U(v)v\,dv
        + O\bigg( \frac{N^{-1+5\nu}}{T} 
        + N^{-\frac{3}{2}+\frac{11}{2}\nu}
        + \frac{T^{1-\frac{\nu}{2}}}{N^2} \bigg) \bigg\} 
    \end{split}\end{equation}
    with $H^*=\frac{N^{2/(1-\nu)}}{T^{2(1-\nu)}}$. We perform the change of variable $2\pi Nv/T=y$, and we note that (if $n=q^a$)
    $$\frac{nN^2}{\Lambda(n)} \sum_{m>N} \frac{\Lambda(m)\Lambda(m/n)}{m^3} 
    =\frac{T^{2\alpha}}{\Lambda(n)} \sum_{b>\alpha\frac{\log T}{\log q}} \frac{\Lambda(n)^2}{q^{3b}}
    \ll \frac{T^{2\alpha} \log q}{q^3}  \frac{1}{T^{3\alpha}}
    \ll T^{-\alpha} \ll \frac{1}{\sqrt{T}}.$$
    Doing so, we obtain
    \begin{equation}\begin{split}\notag
        F_n^{\mathcal A^*\overline{\mathcal B^*}}(\alpha;\psi) 
        = &-\frac{2}{\Lambda(n)} \int_{0}^{1} \bigg( \sum_{1\leq h\leq H^*} \frac{\mathfrak{S}_n(h)}{h^2} - \int_{y}^{H^*}\frac{dx}{x^2} \bigg) \Re\hat\psi_U \bigg(\frac{yT}{2\pi N} \bigg) y\,dy \\
        & -\frac{2}{\Lambda(n)} \int_{1} ^{H^*} \bigg( \sum_{y\leq h\leq H^*} \frac{\mathfrak{S}_n(h)}{h^2} - \int_{y}^{H^*}\frac{dx}{x^2} \bigg) \Re\hat\psi_U \bigg(\frac{yT}{2\pi N} \bigg) y\,dy \\
        &+O\bigg( \frac{nN^{1+5\nu}}{T^2} + \frac{nN^{\frac{1}{2}+\frac{11}{2}\nu}}{T} + nT^{-\frac{\nu}{2}} 
        \bigg).
    \end{split}\end{equation}
    Trivially, one sees that
        \begin{equation}\begin{split}\notag
        &\int_{0}^{1} \bigg( \sum_{1\leq h\leq H^*} \frac{\mathfrak{S}_n(h)}{h^2} - \int_{y}^{H^*}\frac{dx}{x^2} \bigg) \Re\hat\psi_U \bigg(\frac{yT}{2\pi N} \bigg) y\,dy 
        \ll \int_{0}^{1} \bigg( 1+\frac{1}{H^*}+\frac{1}{y}\bigg) y\,dy 
         \ll 1.
    \end{split}\end{equation}
    Finally, using
    \begin{equation}\begin{split}\notag
        \sum_{y\leq h\leq H^*} \frac{\mathfrak{S}_n(h)}{h^2} - \int_{y}^{H^*}\frac{dx}{x^2}
        = T_{2,n}(y) + O\bigg( \frac{(\log H^*)^{2/3}}{(H^*)^{2}} \bigg)
    \end{split}\end{equation}
    and
    \begin{equation}\begin{split}\notag
        T_{2,n}(y) \ll \frac{(\log y)^{2/3}}{y^2},
    \end{split}\end{equation}    
    together with the fast decay of $\hat\psi_U$, we deduce
    \begin{equation}\begin{split}\label{25sep.4}
        F_n^{\mathcal A^*\overline{\mathcal B^*}}(\alpha;\psi) 
        = &-\frac{2}{\Lambda(n)} \int_{1}^{\infty} T_{2,n}(y) \times \Re\hat\psi_U \bigg(\frac{yT}{2\pi N} \bigg) y\,dy\\
        &+O\bigg( \frac{nN^{1+5\nu}}{T^2} + \frac{nN^{\frac{1}{2}+6\nu}}{T} + nT^{-\frac{\nu}{2}} \bigg)
        + O\bigg(\frac{1}{\Lambda(n)}\bigg).
    \end{split}\end{equation}
    Plugging \eqref{25sep.2},\eqref{25sep.3}, and \eqref{25sep.4} into \eqref{25sep.0}, we get
    \begin{equation}\begin{split}\notag
        F_n(\alpha;\psi) 
        =  -\frac{2}{\Lambda(n)} \int_{1}^{\infty} S_{2,n}(y)\times &\Re\hat\psi\bigg(\frac{uT}{2\pi N}\bigg) \frac{dy}{y^3} 
        -\frac{2}{\Lambda(n)} \int_{1}^{\infty} T_{2,n}(y) \times \Re\hat\psi_U \bigg(\frac{yT}{2\pi N} \bigg) y\,dy \\
        &+ O\bigg(\frac{1}{\Lambda(n)}\bigg) 
        +O\bigg( \frac{nN^{1+5\nu}}{T^2} + \frac{nN^{\frac{1}{2}+\frac{11}{2}\nu}}{T} + nT^{-\frac{\nu}{2}} \bigg)
        + O(\mathcal S)
    \end{split}\end{equation}
    as desired.
\end{proof}

\begin{lemma}\label{Lemma9.3}
Assume RH and Conjecture \ref{HLC}. If $n$ is a prime power smaller than $\leq T^{1/2-\varepsilon}$ for some fixed $0<\varepsilon<\frac{1}{2}$, and $T\leq N\leq T^2$,
then the error term $\mathcal S$ defined in \eqref{defErrorE} 
is $\ll T^{-1/4}$.
\end{lemma}

\begin{proof}
By applying Cauchy-Schwarz inequality, we have
\begin{equation}\begin{split}\notag
   \mathcal S \ll \frac{n\log T}{N\sqrt{T}}\bigg(\int_{\mathbb R}\Big(n|\mathcal A(t)|^2 + n|\mathcal A^*(t)|^2 + |\mathcal B(t)|^2 + |\mathcal B^*(t)|^2\Big) \psi_U\bigg(\frac{t}{T}\bigg) dt\bigg)^{1/2} + T^{-1/4} .
\end{split}\end{equation}
We will bound each term above separately, by using the same technique as in the proof of Lemma \ref{LemmaHLC}. The calculation for the first two terms can also be found in \cite[Example 4]{GG} and with more details in \cite[Section 7]{GGOS}. We also remark that the second moments above can be even evaluated asymptotically; however, an upper bound is enough for our purposes. 
As for the second moment of $\mathcal A(t)$, \cite[Corollary 1]{GG} yields
\begin{equation}\label{25nov.1}
   \int_{\mathbb R}|\mathcal A(t)|^2 \psi_U\bigg(\frac{t}{T}\bigg) dt
   = 2T \int_{1}^{\infty} \bigg( \sum_{h\leq y} \mathfrak{S}(h)h^2 - \frac{y^3}{3} \bigg) \Re\hat\psi_U\bigg(\frac{yT}{2\pi N}\bigg) \frac{dy}{y^3}
   +O(T^{1+\varepsilon}).
\end{equation}
since the diagonal term from Goldston-Gonek's formula is 
$$\frac{T\hat\psi_U(0)}{N^2} \sum_{m\leq N} \Lambda(m)^2m 
\ll T^{1+\varepsilon}.$$
Note that Conjecture \ref{HLC} for $n=1$ guarantees that the assumptions of \cite[Corollary 1]{GG} are satisfied.
Moreover, the term involving the singular series in \eqref{25nov.1} can be evaluated with the same strategy as in the proof of \ref{LemmaHLC}, or as in \cite[p. 191]{GG}. Doing so, one gets
\begin{equation}\notag
   \frac{n\log T}{N\sqrt{T}}\bigg(\int_{\mathbb R} n |\mathcal A(t)|^2 \psi_U\bigg(\frac{t}{T}\bigg) dt\bigg)^{1/2}
   \ll \frac{n\log T}{N\sqrt{T}}\sqrt{nT^{1+\varepsilon}}
   \ll \frac{n^{3/2}T^{\varepsilon}}{N} 
   \ll T^{-1/4}.
\end{equation}
With little modification to the previous argument, one also shows that
\begin{equation}\notag
   \frac{n\log T}{N\sqrt{T}}\bigg(\int_{\mathbb R}  |\mathcal B(t)|^2 \psi_U\bigg(\frac{t}{T}\bigg) dt\bigg)^{1/2}
   \ll T^{-1/4}.
\end{equation}
The second moment of $\mathcal A^*$ and $\mathcal B^*$ can be obtained by the same strategy, invoking Corollary 2 of \cite{GG} in place of Corollary 1. The claim follows.
\end{proof}


\begin{lemma}\label{Lemma9.3bis}
    Let $n=q^a$ be a prime power, and $Q$ a parameter. Then,
    \begin{equation}\begin{split}\notag
        -\frac{2}{\Lambda(n)} \int_1^\infty f_n(y) \Re\hat\psi_U\left( \frac{y}{2 \pi Q} \right) \, dy 
        = \min\bigg\{1,\frac{\log Q}{\Lambda(n)} \bigg\}
        &+ O\bigg(\frac{\log Q}{U}\bigg)
        + O\bigg(\frac{1+\log U}{\Lambda(n)} \bigg) .
    \end{split}\end{equation}
\end{lemma}

\begin{proof}
    Abbreviating the left-hand side as LHS and applying Lemma \ref{LemmaAsForf_n}, we obtain
    \begin{equation}\begin{split}\notag
        \mathrm{LHS} 
        = &\frac{1}{\Lambda(n)} \int_1^q  \Re\hat\psi_U\left( \frac{y}{2 \pi Q} \right) \frac{dy}{y}
        +O\bigg(\frac{1}{\Lambda(n)} \int_1^q \frac{1}{y^{5/4}} 
        \bigg|\hat\psi_U\left( \frac{y}{2 \pi Q} \right)\bigg| \, dy\bigg)\\
        &+O\bigg(\frac{1}{q\Lambda(n)} \int_1^q  \bigg|\hat\psi_U\left( \frac{y}{2 \pi Q} \right) \, dy\bigg|\bigg)
        +O\bigg(\frac{q^{1/4}}{\Lambda(n)} \int_q^\infty \frac{1}{y^{5/4}} \bigg|\hat\psi_U\left( \frac{y}{2 \pi Q} \right)\bigg| dy\bigg).
    \end{split}\end{equation}
    Since $\hat\psi_U(x)\ll 1$, the error terms can be bounded trivially, and the above reads
    \begin{equation}\begin{split}\notag
        \mathrm{LHS}
        = &\frac{1}{\Lambda(n)} \int_1^q  \Re\hat\psi_U\left( \frac{y}{2 \pi Q} \right) \frac{dy}{y}
        +O\bigg(\frac{1}{\Lambda(n)}\bigg).
    \end{split}\end{equation}
    As for the main term, we split into two cases. If $q\leq Q$, we Taylor expand $\hat\psi_U$ around 0, and write
    \begin{equation}\label{TaylorExpansionPsi}
        \Re\hat\psi_U\bigg(\frac{y}{2\pi Q}\bigg) 
        = \Re\hat\psi_U(0) + O\bigg(\frac{y}{Q}\bigg)
        = 1 + O\bigg(\frac{1}{U}\bigg) + O\bigg(\frac{y}{Q}\bigg),
    \end{equation}
    since $\Re\hat\psi_U(0) = 1+O(1/U)$. Hence, for $q\leq Q$,
    \begin{equation}\begin{split}\notag
        \mathrm{LHS}
        &= \frac{1}{\Lambda(n)} \int_1^q  \frac{dy}{y}
        +O\bigg(\frac{1}{U\Lambda(n)} \int_1^q \frac{dy}{y} \bigg)
        + O\bigg(\frac{T}{N\Lambda(n)} \int_1^q  dy \bigg)
        +O\bigg(\frac{1}{\Lambda(n)}\bigg)\\
        &= 1
        +O\bigg(\frac{1}{U} \bigg)
        +O\bigg(\frac{1}{\Lambda(n)}\bigg).
    \end{split}\end{equation}
    In the case $q>Q$, we truncate the integral at height $Q$ as follows:
    \begin{equation}\begin{split}\notag
        \mathrm{LHS}
        = \frac{1}{\Lambda(n)} \int_1^{Q}  \Re\hat\psi_U\left( \frac{y}{2 \pi Q} \right) \frac{dy}{y} 
        &+ O\bigg(\frac{1}{\Lambda(n)} \int_{Q}^{QU}  \bigg|\hat\psi_U\left( \frac{y}{2 \pi Q} \right)\bigg| \frac{dy}{y} \bigg) \\
        &+ O\bigg(\frac{1}{\Lambda(n)} \int_{QU}^q  \bigg|\hat\psi_U\left( \frac{y}{2 \pi Q} \right) \bigg| \frac{dy}{y} \bigg)
        +O\bigg(\frac{1}{\Lambda(n)}\bigg).
    \end{split}\end{equation}
    The first term can be evaluated by using Equation \eqref{TaylorExpansionPsi}. For the error terms, we use $\hat\psi_U(x) \ll \min\{1,U/x\}$, and get
    \begin{equation}\begin{split}\notag
        \mathrm{LHS}
        = \frac{\log Q}{\Lambda(n)} 
        &+ O\bigg(\frac{\log Q}{U\Lambda(n)}\bigg)
        + O\bigg(\frac{\log U}{\Lambda(n)} \bigg) 
        +O\bigg(\frac{1}{\Lambda(n)}\bigg)
    \end{split}\end{equation}
    for $q\leq Q$. The claim follows.
\end{proof}

Putting together the previous lemmas, we now prove Proposition \ref{TPC_IntermediateAlpha}.

\begin{proof}[Proof of Proposition \ref{TPC_IntermediateAlpha}]
    We start by Lemma \ref{LemmaHLC}. A bound for $\mathcal S$ is given by Lemma \ref{Lemma9.3}. As for the main term, we apply Lemma \ref{Lemma9.3bis} with $Q=N/T$. This yields
    \begin{equation}\begin{split}\notag
        F_n(\alpha;\psi_U) = \min\bigg\{1,\frac{\log (N/T)}{\Lambda(n)} \bigg\}
        &+ O\bigg(\frac{\log (N/T)}{U}\bigg)
        + O\bigg(\frac{1+\log U}{\Lambda(n)} \bigg) 
        + O(T^{-1/4})\\
        &+ O\bigg(\frac{1}{\Lambda(n)}\bigg) 
        +O\bigg( \frac{nN^{1+5\nu}}{T^2} + \frac{nN^{\frac{1}{2}+\frac{11}{2}\nu}}{T} + nT^{-\frac{\nu}{2}} \bigg),
    \end{split}\end{equation}
    We take $U=(\log T)^2$ so that 
    $$ \frac{\log (N/T)}{U} + \frac{1+\log U}{\Lambda(n)} \ll \frac{\log\log T}{\Lambda(n)}. $$
    If
    \begin{equation}\label{18Nov.1}
    T\leq N\leq T^{\frac{1-\nu/2}{1/2+11\nu/2}} 
    \quad \text{i.e.} \quad 
    1-\frac{\log n}{\log T} \leq \alpha \leq \frac{1-\nu/2}{1/2+11\nu/2}- \frac{\log n}{\log T},
    \end{equation}
    the two error terms $nN^{1+5\nu}/T^2$ and $nN^{1/2+11\nu/2}/T$ are both $\ll nT^{-\nu/2}$. Hence, 
    \begin{equation}\begin{split}\notag
        F_n(\alpha;\psi_U) = \min\bigg\{1,\frac{\log (N/T)}{\Lambda(n)} \bigg\}
        &+ O\bigg(\frac{\log\log T}{\Lambda(n)}\bigg) 
        +O( nT^{-\frac{\nu}{2}} ).
    \end{split}\end{equation}
    To control the last error term, it suffices to take $\nu=(2+\varepsilon)\frac{\log n}{\log T}$ with $\varepsilon = \frac{1}{100}$, say. This guarantees that $nT^{-\nu/2}\ll n^{-\varepsilon/2} \ll 1/\log n$. With these choices, the range in \eqref{18Nov.1} becomes
    $$
    1-\frac{\log n}{\log T} \leq \alpha \leq \frac{2-(2+\varepsilon)\frac{\log n}{\log T}}{1+11(2+\varepsilon)\frac{\log n}{\log T}}- \frac{\log n}{\log T},$$
    Since the range above is contained in the range in \eqref{RangeIntermediatealpha},
    this concludes the proof.
    \end{proof}
    
    \begin{remark}\label{Remark48}
    We note that the most restrictive error term in terms of the range of $n$ and $N$ considered in Proposition \ref{TPC_IntermediateAlpha} is given by the last error term of \cite[Corollary 2]{GG}. It seems plausible that one could make this error term smaller by choosing the parameter $H^*$ there bigger. This in turn is possible as long as the parameter $\eta$, which controls the uniformity in the shift $h$, is strictly bigger than $\frac{1}{2}$, compare \cite[p. 174]{GG}. We have assumed this to be the case through Conjecture \ref{HLC}. Following this approach, one may be able to improve upon Proposition \ref{TPC_IntermediateAlpha}. In particular, this may allow one to choose $\alpha$ up to essentially $2 - 3 \frac{\log n}{\log T}$ on a range of $n \le T^\beta$ for some fixed $\beta > 0$.
    \end{remark}

\bibliography{lit}
\bibliographystyle{alpha}

\end{document}